\documentclass[reqno, 12pt,a4letter]{amsart}
\usepackage{amsmath,amsxtra,amssymb,latexsym, amscd,amsthm}
\usepackage{graphicx,color}
\usepackage{subfigure}
\usepackage[utf8]{inputenc}
\usepackage{enumerate}
\usepackage[mathscr]{euscript}
\usepackage{mathrsfs}
\usepackage[english]{babel}
\usepackage[euler]{textgreek}
\usepackage{cite}
\newtheorem {theorem}{Theorem}[section]
\newtheorem {corollary}[theorem]{Corollary}
\newtheorem {proposition}[theorem]{Proposition}
\newtheorem {lemma}[theorem]{Lemma}

\newtheorem {example}[theorem]{Example}
\newtheorem {definition}[theorem]{Definition}
\newtheorem {remark}[theorem]{Remark}

\def\ar{a\kern-.370em\raise.16ex\hbox{\char95\kern-0.53ex\char'47}\kern.05em}
\def\ees{{\accent"5E e}\kern-.385em\raise.2ex\hbox{\char'23}\kern-.08em}
\def\eex{{\accent"5E e}\kern-.470em\raise.3ex\hbox{\char'176}}
\def\AR{A\kern-.46em\raise.80ex\hbox{\char95\kern-0.53ex\char'47}\kern.13em}
\def\EES{{\accent"5E E}\kern-.5em\raise.8ex\hbox{\char'23 }}
\def\EEX{{\accent"5E E}\kern-.60em\raise.9ex\hbox{\char'176}\kern.1em}
\def\ow{o\kern-.42em\raise.82ex\hbox{
  \vrule width .12em height .0ex depth .075ex \kern-0.16em \char'56}\kern-.07em}
\def\OW{O\kern-.460em\raise1.36ex\hbox{
\vrule width .13em height .0ex depth .075ex \kern-0.16em \char'56}\kern-.07em}
\def\UW{U\kern-.42em\raise1.36ex\hbox{
\vrule width .13em height .0ex depth .075ex \kern-0.16em \char'56}\kern-.07em}

\def\B {\mathbb{B}}

\title{Clarke's tangent cones, subgradients, optimality conditions and the Lipschitzness at infinity}

\author{MINH T\`UNG NGUY\EEX N} 
\address[Minh-Tung Nguyen]{Faculty of Mathematical Economics, Banking University of Ho Chi Minh City, Ho Chi Minh City,Vietnam}
\email{tungnm@hub.edu.vn}

\author{TI\EES N-S\OW N PH\d{A}M}
\address[Ti\ees n-S\ow n Ph\d{a}m]{Department of Mathematics, Dalat University, 1 Phu Dong Thien Vuong, Dalat, Vietnam}
\email{sonpt@dlu.edu.vn}

\date{ \today}

\thanks{{\bf Funding}: The second author is supported by Vietnam National Foundation for Science and Technology Development (NAFOSTED) grant 101.04-2023.06}

\begin{document}
\maketitle

\begin{abstract}
We first study Clarke's tangent cones at infinity to unbounded subsets of $\mathbb{R}^n.$ We prove that these cones are closed convex and show a characterization of their interiors. We then study subgradients at infinity for extended real value functions on $\mathbb{R}^n$ and derive necessary optimality conditions at infinity for optimization problems. We also give a number of rules for the computing of subgradients at infinity and provide some characterizations of the Lipschitz continuity at infinity for lower semi-continuous functions.
\end{abstract}

\section{Introduction} \label{Section1}

The theory of subgradients (or generalized gradients) is concerned with the differential properties of functions, which do not have derivatives in the usual sense. This theory, in the setting of finite dimensional spaces, associates with a function $f \colon \mathbb{R}^n \to \overline{\mathbb{R}} := \mathbb{R} \cup \{\infty\}$ and a point $x \in \mathbb{R}^n$, a (possibly empty) closed subset $\partial f(x)$ of $\mathbb{R}^n$ whose elements are called {\em subgradients} of $f$ at $x.$ The theory plays an important role in many areas such as optimization, optimal control and differential equations; for more details, we refer the reader to \cite{Clarke1975, Clarke1975-1, Clarke1976, Clarke1979, Clarke1981, Clarke1990, Clarke1998, Hiriart-Urruty1979, Ioffe1981, Ioffe2017, Kim2021, Ledyaev2007, Mordukhovich2018, Mordukhovich2019, Rockafellar1979, Rockafellar1980, Rockafellar1998}.

Rockafellar \cite{Rockafellar1963} defined $\partial f(x)$ for convex functions,
showed how to characterize $\partial f(x)$ in terms of one-sided directional derivatives, and proved that under certain assumptions, rules such as the necessary optimality condition $0 \in \partial f(x)$ are valid.

Clarke \cite{Clarke1973} showed how the definition of $\partial f(x)$ could be extended to arbitrary lower semi-continuous functions $f$ in such a way that $\partial f(x)$ is the set of subgradients in the sense of convex analysis when $f$ is convex, and $\partial f(x)$ reduces to the gradient vector $\nabla f(x)$ when $f$ is smooth (continuously differentiable).
He expressed $\partial f(x)$ by way of a generalized directional derivative when $f$ is Lipschitz continuous in a neighborhood of $x.$ 
He used it to derive necessary optimality conditions for non-smooth, non-convex problems in optimal control and mathematical programming. For locally Lipschitz functions, he proved that $\partial f(x)$ is nonempty compact and provided a number of rules for subgradient calculation that generalize the ones previously known for convex functions.

Rockafellar \cite{Rockafellar1978, Rockafellar1979, Rockafellar1980} supplied an alternative development of Clarke's ideas that includes a direct definition of $\partial f(x)$ in terms of a still more general directional derivative function. He expressed $\partial f(x)$ by way of the normal cone of the epigraph of $f$ at $(x, f(x)).$ He proved that a lower semi-continuous function $f$ is Lipschitz in a neighborhood of a point $x$ if and only if $\partial f(x)$ is nonempty compact. He also extended Clarke's results to functions that are not necessarily locally Lipschitz.

Following the ideas of Clarke and Rockafellar, we define and study {\em subgradients at infinity} of extended real value functions.

\subsection*{Motivations}
Our first motivation comes from optimality conditions in mathematical programming. Namely, let $f \colon \mathbb{R}^n \to \overline{\mathbb{R}}$ be bounded below, i.e.,
\begin{eqnarray*}
f_* &:=& \inf_{x \in \mathbb{R}^n} f(x) \ > \ -\infty.
\end{eqnarray*}
If $f$ has a global minumum at a point $x \in \mathbb{R}^n,$ then certainly $0 \in \partial f(x).$ In general, the function $f$ may not attain its infimum, and so there exists a sequence $x_k$ tending to infinity such that $f(x_k)$ tends to $f_*;$ then it is natural to find necessary optimality conditions {\em at infinity} for such cases. 

Our second motivation is to characterize the Lipschitzness at infinity for lower semi-continuous functions.

\subsection*{Contributions}
With notation and definitions given in the next sections, our main contributions are as follows:
\begin{itemize}
\item Given an unbounded set $C \subset \mathbb{R}^n$ and an index set $I \subset \{1, \ldots, n\},$ we define the Clarke tangent cone $T_C(\infty_I)$ and normal cone $N_C(\infty_I)$ to $C$ at infinity. We show that $T_C(\infty_I)$ is a closed convex cone and give a characterization of its interior
(see Corollary~\ref{Corollary35} and Theorem~\ref{Theorem36}).

\item Given a function $f \colon \mathbb{R}^n \to \overline{\mathbb{R}},$ we define the set $\partial f(\infty)$ of {\em subgradients of $f$ at infinity} in terms of the normal cone to the epigraph of $f$ at infinity. Then we present a number of rules for the computing of subgradients at infinity. In particular, we demonstrate in Theorem~\ref{Theorem414} that under mild assumptions, if the restriction of $f$ on an unbounded closed set $C \subset\mathbb{R}^n$ does not attain its infimum, then the following {\em necessary optimality condition at infinity} holds true: $0 \in \partial f(\infty) + N_C(\infty_I),$ where $I := \{1, \ldots, n\}.$

\item We show in Theorem~\ref{Theorem57} that a lower semi-continuous function $f \colon \mathbb{R}^n \to \overline{\mathbb{R}}$ is {\em Lipschitz at infinity} if and only if $\partial f(\infty)$ is nonempty compact.
\end{itemize}

It should be emphasized, due to the unboundedness of neighborhoods at infinity, that smooth functions are not necessarily Lipschitz at infinity and that the set of subgradients of $f$ at infinity may not be a singleton set when $f$ is smooth and Lipschitz at infinity (see Example~\ref{Example58} below).

We also note that some results given in this paper can be extended to infinite dimensional spaces. However, to lighten the exposition, we do not pursue this extension here.

The rest of the paper is organized as follows. In Section~\ref{Section2}, some notation and definitions are given.
In Section~\ref{Section3} we study Clarke's tangent and normal cones at infinity to unbounded subsets of $\mathbb{R}^n.$
In Section~\ref{Section4} we give the definition of subgradients at infinity for extended real valued functions on $\mathbb{R}^n$ and establish some properties of them. In particular, we provide necessary optimality conditions at infinity for optimization problems.
In Section~\ref{Section5}, some characterizations of the Lipschitz continuity at infinity for lower semi-continuous functions are given.

\section{Preliminaries}\label{Section2}

\subsection{Notation} 
In this paper we deal with the Euclidean space $\mathbb{R}^n$ equipped with the usual scalar product $\langle \cdot, \cdot \rangle$ and the corresponding norm $\| \cdot\|.$ We denote by $\mathbb{B}_r(x)$ the closed ball centered at $x$ with radius $r;$  when ${x}$ is the origin of $\mathbb{R}^n$ we write $\mathbb{B}_{r}$ instead of $\mathbb{B}_{r}({x}),$ and when $r = 1$ we write  $\mathbb{B}$ instead of $\mathbb{B}_{1}.$ For any two different points $x$ and $x'$ in $\mathbb{R}^n,$ the {\em open line segment} joining $x$ and $x'$ is the set
\begin{eqnarray*}
(x, x') &:=& \{(1 - t) x + tx' \mid  0 < t < 1 \}.
\end{eqnarray*}

We will adopt the convention that 
$\inf \emptyset = +\infty$ and $\sup \emptyset = -\infty,$ and that $\lambda_1 + \lambda_2 = +\infty$ if either $\lambda_1$ or $\lambda_2$ is $+\infty$ (even if the other is $-\infty$). It also is expedient to set
\begin{eqnarray*}
(\pm \infty) + \lambda &=& \lambda + (\pm \infty) \ = \  \pm \infty \quad \textrm{ for any real } \lambda, \\
\lambda \cdot (\pm \infty) &=& (\pm \infty)  \cdot \lambda \ = \ \pm \infty \quad \textrm{ for all } \lambda > 0, \\
\lambda \cdot (\pm \infty) &=& (\pm \infty)  \cdot \lambda \ = \ \mp \infty \quad \textrm{ for all } \lambda < 0.
\end{eqnarray*}
The order on $\mathbb{R} \cup \{\pm \infty\}$ is extended in the natural way to
$$-\infty < r < +\infty$$
for all $r \in \mathbb{R}.$

Let $C$ be a nonempty subset of $\mathbb{R}^n.$ The closure, interior, boundary, and convex hull of the set $C \subset \mathbb{R}^n$ will be written as $\mathrm{cl}{C},$ $\mathrm{int} C,$ $\partial C,$ and $\mathrm{co} C,$  respectively. The distance from a point $x \in \mathbb{R}^n$ to $C$ is defined by $d_C(x) :=\inf_{y \in C}\|x - y\|.$ By definition, it is easy to see that the distance function 
$$d_C \colon \mathbb{R}^n \to \mathbb{R}, \quad x \mapsto d_C(x),$$
is nonnegative and $1$-Lipschitz.

\subsection{Clarke's tangent cones and generalized gradients} 
In this subsection we recall some definitions, which can be found in~\cite{Clarke1990, Rockafellar1970, Rockafellar1978, Rockafellar1998}.

Let $C$ be a subset of $\mathbb{R}^n.$ For each $x \in C,$ the {\em tangent cone} $T_C(x)$ in the sense of Clarke consists of all 
vectors $v \in \mathbb{R}^n$ such that, whenever we have sequences $x_k \in C$ with $x_k \to x$ and $t_k \searrow 0,$ there exists a sequence $v_k \to v$ such that $x_k + t_k v_k \in C$ for all $k.$  It is well-known that $T_C(x)$ is a closed convex cone in $\mathbb{R}^n$ containing the origin $0 \in \mathbb{R}^n$ (see \cite{Clarke1990, Rockafellar1978, Rockafellar1980}). Its polar 
\begin{eqnarray*}
N_C(x) &:=& \{w \in \mathbb{R}^n \ | \ \langle v, w \rangle \le 0 \quad \textrm{ for all } \quad v \in T_C(x)\}
\end{eqnarray*}
is called the {\em normal cone} to $C$ at $x.$ Clearly, $N_C(x)$ is also a closed convex cone.

Let $f$ be an extended real valued function on $\mathbb{R}^n.$ As usual, 
the effective domain and {\em epigraph} of $f$ are denoted by, respectively,
\begin{eqnarray*}
\mathrm{dom} f &:=&  \{ x \in \mathbb{R}^n \mid f(x) < \infty \}, \\
\mathrm{epi} f &:=&  \{ (x, y) \in \mathbb{R}^n \times \mathbb{R} \mid f(x) \le y \}.
\end{eqnarray*}
The {\em generalized (Clarke) gradient} of $f$ at $x \in \mathbb{R}^n$ with $f(x)$ finite is defined by 
\begin{eqnarray*}
\partial f(x) &:=& \{\xi \in \mathbb{R}^n \ | \ (\xi, -1) \in N_{\textrm{epi} f}(x, f(x))\}.
\end{eqnarray*}
If $f(x) = \pm \infty,$ we set $\partial f(x) := \emptyset.$  It is well-known that 
(see \cite[Propositions~2.2.4 and 2.2.7]{Clarke1990} and \cite[Theorem~5]{Rockafellar1980})
when $f$ is continuously differentiable, $\partial f(x)$ reduces to the singleton set $\{\nabla f(x)\},$ and when $f$ is convex and is finite at $x,$ then $\partial f(x)$ coincides with what is called the {\em subgradient set} of convex analysis; that is, the set of vectors $\xi$ in $\mathbb{R}^n$ satisfying
\begin{eqnarray*}
f(x') - f(x) &\ge& \langle \xi, x' - x \rangle \quad \textrm{ for all } \quad x' \in \mathbb{R}^n.
\end{eqnarray*}

\section{Clarke's tangent cones at infinity} \label{Section3}

Let $C$ be a subset of $\mathbb{R}^n$ and $I$ be a nonempty subset of $\{1, \ldots, n\}.$ Consider the projection $\pi \colon \mathbb{R}^n \to \mathbb{R}^{\#I}, x := (x_1, \ldots, x_n) \mapsto (x_i)_{i \in I}.$ For each $R > 0$ let
\begin{eqnarray*}
\B_{R, I} &:=& \{x \in \mathbb{R}^n \ | \ \|\pi(x)\| \le R \}.
\end{eqnarray*}
Assume that the set $\pi(C)$ is unbounded. 

\begin{definition}{\rm
(i) By the {\em Clarke tangent cone of $C$ at infinity (with respect to the index set $I$)}, denoted $T_C(\infty_I),$  we mean the set of all vectors $v \in \mathbb{R}^n$ such that, whenever we have sequences $x_k \in C$ with $\pi(x_k) \to \infty$ and $t_k \searrow 0,$ there exists a sequence $v_k \to v$ such that $x_k + t_k v_k \in C$ for all $k.$

(ii) By the {\em normal cone to $C$ at infinity (with respect to the index set $I$)}, we mean the set
\begin{eqnarray*}
N_C(\infty_I) &:=& \{w \in \mathbb{R}^n \ | \ \langle v, w \rangle \le 0 \quad \textrm{ for all } \quad v \in T_C(\infty_I)\}.
\end{eqnarray*}
}\end{definition}

\begin{remark}{\rm 
For an extended real valued function $f$ on $\mathbb{R}^n,$ the definition of $\partial f(\infty),$ which relates to the variation of $f(x)$ as $x$ tends to infinity, will be given in terms of the normal cone to the epigraph of $f$ at infinity. Hence, we will be only interested in points $(x, y) \in \mathrm{epi} f \subset \mathbb{R}^n \times \mathbb{R}$ as $x$ goes to infinity, i.e., points in the epigraph of $f$ whose $i$-th coordinates tend to infinity for some  $i$ in the index set $\{1, \ldots, n\} \subset \{1, \ldots, n, n + 1\}.$ This explains why the index set $I$ appears in the above definition of tangent and normal cones at infinity.
}\end{remark}

\begin{lemma} \label{Lemma33} 
For a vector $v \in \mathbb{R}^n,$ the following statements are equivalent:
\begin{enumerate}[\rm (i)]
\item $v \in T_C(\infty_{I}).$

\item For any sequences $x_k \in C$ with $\pi(x_k) \to \infty$ and $t_k \searrow 0,$ there exists a sequence $v_k \to v$ such that $x_k + t_k v_k \in C$ for infinitely many $k.$

\item For every $\epsilon > 0$ there exist constants $R > 0$ and $\lambda > 0$ such that
\begin{eqnarray} \label{PT21}
C \cap [{x} + t \B_{\epsilon}(v)] &\ne& \emptyset \quad \textrm{ for all } \quad 
{x} \in C \setminus \B_{R, I} \ \textrm{ and } \ t \in [0, \lambda].
\end{eqnarray} 

\item $\lim_{\pi(x) \rightarrow \infty, \, x \in C, \, t \searrow 0}\dfrac{d_C(x+tv)}{t} = 0.$
\end{enumerate}
\end{lemma}

\begin{proof}
(i) $\Rightarrow$ (ii). This is clear from the definition of $T_C(\infty_{I}).$

(ii) $\Rightarrow$ (iii). Suppose the implication were false. Then we could find a real number $\epsilon > 0$ and sequences $x_k \in C$ with $\pi(x_k) \to \infty$ and $t_k \searrow 0,$ such that 
\begin{eqnarray*}
C \cap [{x_k} + t_k \B_{\epsilon}(v)] &=& \emptyset \quad \textrm{ for all } \quad k.
\end{eqnarray*} 
This contradicts our assumption that there is a sequence $v_k \to v$ such that $x_k + t_k v_k \in C$ for infinitely many $k.$

(iii) $\Rightarrow$ (iv). Take an arbitrary $\epsilon > 0.$ There exist constants $R > 0$ and $\lambda > 0$ such that for each ${x} \in C \setminus \B_{R, I}$ and $ t \in [0, \lambda]$ we have $C \cap [{x} + t \B_{\epsilon}(v)] \neq \emptyset.$ For such $x$ and $t,$ there is 
$\overline{x} \in C \cap [{x} + t \B_{\epsilon}(v)],$ and so
\begin{eqnarray*} 
\dfrac{d_C(x+tv)}{t} & \le & \dfrac{\| x + tv - \overline{x}\|}{t} \  \le \ \epsilon. 
\end{eqnarray*}
Since $\epsilon > 0$ is chosen arbitrarily small, it holds that
\begin{eqnarray*} 
\lim_{\pi(x) \rightarrow \infty, \, x \in C, \, t \searrow 0}\dfrac{d_C(x+tv)}{t} &=& 0.
\end{eqnarray*}

(iv) $\Rightarrow$ (i). Take any sequences $x_k \in C$ with $\pi(x_k) \rightarrow \infty$ and $t_k \searrow 0.$ We have 
\begin{eqnarray*}
\epsilon_k &:=& \dfrac{d_C(x_k+t_kv)}{t_k} \ \to \ 0
\end{eqnarray*}
and there is $\overline{x}_k \in C$ such that 
\begin{eqnarray*}
\frac{\|x_k+t_kv-\overline{x}_k \|}{t_k} & < & \epsilon_k +\dfrac{1}{k}. 
\end{eqnarray*}
Setting $v_k:=\dfrac{\overline{x}_k-x_k}{t_k},$ we get $\|v_k -v\| < \epsilon_k +\dfrac{1}{k}$. Hence, $v_k \to v$ and $x_k+t_kv_k = \overline{x}_k \in C$ for all $k.$ By definition, $v \in T_C(\infty).$
\end{proof}

\begin{lemma} \label{Lemma34}
Let $D$ be a nonempty compact subset of $T_C(\infty_I).$ Then for every $\epsilon > 0$ there exist constants $R > 0$ and $\lambda > 0$ such that \eqref{PT21} holds simultaneously for all $v \in \mathrm{co} D.$
\end{lemma}

\begin{proof}
Take arbitrary $\epsilon > 0.$ Since $D$ is compact, it can be covered by a finite family of balls $\B_\epsilon(v_i),$ where $v_i \in D$ for $i = 1, \ldots, m.$ Then
\begin{eqnarray*}
\mathrm{co} D &\subset& \mathrm{co}\left( \bigcup_{i = 1}^m \B_\epsilon(v_i) \right) \ = \ \mathrm{co}\{v_1, \ldots, v_m\} + \B_\epsilon.
\end{eqnarray*} 
Hence, it will suffice to show \eqref{PT21} holds for all $v$ in $\mathrm{co}\{v_1, \ldots, v_m\},$ because this will imply that for all $v \in \mathrm{co}D$ we have
\begin{eqnarray*}
C \cap [{x} + t \B_{2 \epsilon}(v)] &\ne& \emptyset \quad \textrm{ for all } \quad 
{x} \in C \setminus \B_{R, I} \ \textrm{ and } \ t \in [0, \lambda],
\end{eqnarray*} 
which is equivalent to the desired conclusion since $\epsilon > 0$ is arbitrary anyway.

For $i = 1, \ldots, m,$ we have that \eqref{PT21} holds for $v_i$ and certain positive constants $R_i$ and $\lambda_i.$ Letting 
$$\overline{R} := \max_{i = 1, \ldots, m}  R_i \quad \textrm{ and } \quad \overline{\lambda} := \min_{i = 1, \ldots, m} \lambda_i,$$
we have
\begin{eqnarray} \label{PT23}
C \cap [{x} + t \B_{\epsilon}(v_i)] &\ne& \emptyset \quad \textrm{ for all } {x} \in C \setminus \B_{\overline{R}, I}, \ t \in [0, \overline{\lambda}], \ \textrm{ and } \ i = 1, \ldots, m.
\end{eqnarray} 
Let $R \ge \overline{R}$ and $\lambda \in (0, \overline{\lambda}]$ be such that
\begin{eqnarray*}
R - \lambda(\rho + \epsilon) &\ge& \overline{R},
\end{eqnarray*}
where $\rho := \max_{i = 1, \ldots, m} \|v_i\|.$ The assertion
\begin{eqnarray} \label{PT25}
C \cap [{x} + t \B_{\epsilon}(v)] &\ne& \emptyset \quad \textrm{ for all } {x} \in C \setminus \B_{R, I}, \ t \in [0, \lambda], \textrm{ and }
\ v \in \mathrm{co}\{v_1, \ldots, v_k\}
\end{eqnarray} 
holds trivially for $k = 1,$ in view of \eqref{PT23}. Make the induction hypothesis that it holds for $k = m - 1.$ 

Take arbitrary ${x} \in C \setminus \B_{R, I},$ $t \in [0, \lambda],$ and $v \in \mathrm{co}\{v_1, \ldots, v_m\}.$ We can write
$$v = \alpha v' + (1- \alpha) v_m,$$
where $\alpha \in [0, 1]$ and $v' \in \mathrm{co}\{v_1, \ldots, v_{m - 1}\}.$ Since $\alpha t \in [0, \lambda],$ we have by induction that $C$ meets ${x} + \alpha t \B_\epsilon(v').$ Let ${x}'$ be any point in the intersection. Then ${x}' \in {x} + \alpha t \B_\epsilon(v'),$ and so,
\begin{eqnarray*}
\|\pi({x}')\| 
&\ge& \|\pi({x})\| - \alpha t (\|v'\| + \epsilon) \\
&>& R - \lambda (\rho + \epsilon) \ \ge \ \overline{R}.
\end{eqnarray*} 
Thus, ${x}' \in C \setminus \B_{\overline{R}, I}.$ Since $(1 - \alpha) t \le \lambda \le \overline{\lambda},$ it follows from \eqref{PT23} that $C$ meets ${x}' + (1 - \alpha)t \B_\epsilon(v_m).$ Hence $C$ meets
$${x} + \alpha t \B_\epsilon(v') + (1 - \alpha)t \B_\epsilon(v_m) = {x} + t \B_\epsilon(v).$$
This proves \eqref{PT25} for $k = m$ and completes the proof.
\end{proof}

\begin{corollary}\label{Corollary35}
The set $T_C(\infty_I)$ is a closed convex cone in $\mathbb{R}^n$ containing the origin.
\end{corollary}
\begin{proof}
By definition, it is easy to check that $T_C(\infty_I)$ is a closed cone containing the origin. Moreover, it is convex in view of Lemma~\ref{Lemma34}.
\end{proof}

We are ready to prove the first of our main theorems, which gives a characterization of interior points of Clarke's tangent cones at infinity.

\begin{theorem}\label{Theorem36}
Assume that the set $C$ is closed. Then $v \in \mathrm{int} T_C(\infty_I)$ if and only if there exist constants $\epsilon > 0, R > 0,$ and $\lambda > 0$ such that 
\begin{eqnarray} \label{PT31}
{x} + tv' \in C \quad \textrm{ for all } {x} \in C \setminus \B_{R, I}, \ t \in [0, \lambda], \ \textrm{ and } \ v' \in \B_\epsilon(v).
\end{eqnarray} 
\end{theorem}

\begin{proof}
{\em Sufficiency.} Assume that the assertion \eqref{PT31} holds, and take arbitrary $v' \in \B_\epsilon(v).$ For any sequences $x_k \in C$ with 
$\pi(x_k) \to \infty$ and $t_k \searrow 0,$ we have $x_k \not \in \B_{R, I}$ and $t_k \in (0, \lambda)$ for all $k$ sufficiently large, and so $x_k + t_k v_k \in C$ for $v_k := v'.$ By definition, then $v' \in T_C(\infty_I).$ Therefore, $\B_\epsilon(v) \subset T_C(\infty_I).$ Consequently, $v \in \mathrm{int} T_C(\infty_I).$ 

{\em Necessity.} Given $v \in \mathrm{int} T_C(\infty_I),$ choose $\epsilon > 0$ small enough that $\B_{3 \epsilon} (v) \subset T_C(\infty_I).$ Apply Lemma~\ref{Lemma34} to the compact set $D := \B_{3 \epsilon} (v)$ to obtain $R' > 0$ and $\lambda' > 0$ such that 
\begin{eqnarray} \label{PT32}
C \cap [{x} + t\B_\epsilon(v')] \ne \emptyset \quad \textrm{ whenever } \ {x} \in C \setminus \B_{R, I}, \ t \in [0, \lambda'], \textrm{ and } \ v' \in \B_{3 \epsilon} (v).
\end{eqnarray} 
Next choose constants $R$ and $\lambda$ so that
\begin{eqnarray} \label{PT33}
R - 2 \lambda(\|v\| + \epsilon) \ge R' \quad \textrm{ and } \quad \lambda \in (0,  \lambda'].
\end{eqnarray}
We will show that \eqref{PT31} holds for this choice of $\epsilon, R,$ and $\lambda.$

By contradiction, suppose that \eqref{PT31} does not hold. Then there exist
\begin{eqnarray} \label{PT34}
\overline{x} \in C \setminus \B_{R, I}, \quad \overline{\lambda} \in [0, \lambda], \quad \textrm{ and } \quad \overline{v} \in \B_\epsilon(v)
\end{eqnarray}
such that $\overline{x} + \overline{\lambda} \overline{v} \not \in C.$ Since $C$ is closed, we can choose $\rho > 0$ small enough that
\begin{eqnarray} \label{PT35}
C \cap \B_\rho(\overline{x} + \overline{\lambda} \overline{v}) &=& \emptyset.
\end{eqnarray}
In particular, $\rho < \|\overline{x} - (\overline{x} + \overline{\lambda} \overline{v})\| = \overline{\lambda} \|\overline{v} \|.$

Let 
\begin{eqnarray} \label{PT36}
\widetilde{\lambda} 
&:=& \max\{s \in [0, \overline{\lambda}] \ | \ C \cap \B_\rho(\overline{x} + s \overline{v}) \ne \emptyset \} \\
&=& \max\{s \in [0, \overline{\lambda}] \ | \ C \cap [\overline{x} + s \overline{v} + \B_\rho] \ne \emptyset \}; \nonumber
\end{eqnarray}
this maximum is attached because $C$ is closed and $\B_\rho$ is compact. Since $\overline{x} \in C$ and \eqref{PT35} holds, we have
\begin{eqnarray} \label{PT37}
0 < \widetilde{\lambda} < \overline{\lambda} \le \lambda \le \lambda'.
\end{eqnarray}
Select any $\widetilde{x} \in C \cap  \B_\rho(\overline{x} + \widetilde{\lambda} \overline{v}),$ as exists by \eqref{PT36}. The interior of the ball $\B_\rho(\overline{x} + \widetilde{\lambda} \overline{v})$ cannot meet $C,$ in view of \eqref{PT36}. Thus,
\begin{eqnarray*}
\widetilde{x} &=& \overline{x} + \widetilde{\lambda}  \overline{v} + \rho e \quad \textrm{ with} \quad \|e\| = 1.
\end{eqnarray*}
We deduce from \eqref{PT34} and \eqref{PT37} that
\begin{eqnarray*}
\|\pi(\widetilde{x})\| 
&\ge& \|\pi(\overline{x})\| - \|\pi(\widetilde{x}  - \overline{x})\| \\
&>& R - \|\pi(\widetilde{\lambda} \overline{v} + \rho e)\| \\
&\ge& R - \|\widetilde{\lambda} \overline{v} + \rho e \| \\
&\ge& R - \widetilde{\lambda} \| \overline{v}\| - \rho \\
&>& R - \widetilde{\lambda} \| \overline{v}\| - \overline{\lambda} \| \overline{v}\| \\
&\ge& R - 2{\lambda} \| \overline{v}\| \\
&\ge& R - 2{\lambda} (\| {v}\| + \epsilon).
\end{eqnarray*}
It follows from \eqref{PT33} that for this $\widetilde{x}$ and for $\widetilde{v} := \overline{v} - 2 \epsilon e$ we have 
$$\widetilde{x} \in C \setminus \B_{R', I} \quad \textrm{  and } \quad \widetilde{v} \in \B_{3 \epsilon}(v).$$ 
Therefore, by \eqref{PT32},
\begin{eqnarray}\label{PT39}
C \cap [\widetilde{x} + t \B_\epsilon(\widetilde{v}) ] &\ne& \emptyset \quad \textrm{ for all } \quad t \in [0, \lambda'].
\end{eqnarray} 
However, take any $t > 0$ small enough that
\begin{eqnarray}\label{PT310}
0 < t < \min \{\frac{\rho}{2 \epsilon}, \overline{\lambda} - \widetilde{\lambda}\}.
\end{eqnarray} 
(Hence $t < \overline{\lambda} \le \lambda'.)$ It will be shown that
\begin{eqnarray}\label{PT311}
C \cap [\widetilde{x} + t \B_\epsilon(\widetilde{v})] &=& \emptyset.
\end{eqnarray} 
The contradiction between this and \eqref{PT39} will finish the proof.

Since $t < \frac{\rho}{2 \epsilon}$ in \eqref{PT310}, we have
$$0 < \rho - 2 \epsilon t < \rho - \epsilon t < \rho,$$
so that
\begin{eqnarray*}
\widetilde{x} + t\B_{\epsilon}(\widetilde{v}) 
&=& (\overline{x} + \widetilde{\lambda} \overline{v} + \rho e) + t(\overline{v} - 2 \epsilon e) + t \epsilon \B \\
&=& \overline{x} + (\widetilde{\lambda} + t) \overline{v} + (\rho - 2 t \epsilon) e + t \epsilon \B \\
&\subset& \overline{x} + (\widetilde{\lambda} + t) \overline{v} + \B_\rho.
\end{eqnarray*} 
This yields \eqref{PT311}, because 
\begin{eqnarray*}
C \cap [\overline{x} + (\widetilde{\lambda} + t) \overline{v} + \B_\rho] = \emptyset 
\end{eqnarray*} 
by the definition of $\widetilde{\lambda}$ and for $t > 0$ small enough.
\end{proof}

\begin{corollary} \label{Corollary37}
Assume that the set $C$ is closed. If the interior of $T_C(\infty_I)$ is nonempty, then the multifunction $N_C \colon C \rightrightarrows \mathbb{R}^n$ is closed at infinity, in the sense that
\begin{eqnarray*}
x_k \in C, \ \pi(x_k) \to \infty, \ z_k \in N_C(x_k), \ z_k \to z \quad \Longrightarrow \quad z \in N_C(\infty_I).
\end{eqnarray*}
\end{corollary}
\begin{proof}
Take arbitrary $v \in \mathrm{int} T_C(\infty_I).$ In view of Theorem~\ref{Theorem36}, there exists a constant $R > 0$ such that $v \in T_C(x)$ for all $x \in C \setminus \B_{R, I},$ and hence $v \in T_C(x_k)$ for all $k$ sufficiently large. By the definition of $N_C(x_k),$ then $\langle v, z_k \rangle \le 0.$ Consequently,
$\langle v, z \rangle = \lim_{k \to \infty} \langle v, z_k \rangle \le 0.$ Therefore,
\begin{eqnarray*}
\langle v, z \rangle &\le& 0 \quad \textrm{ for all } \quad v \in \mathrm{int} T_C(\infty_I).
\end{eqnarray*}
Since $T_C(\infty_I)$ is convex with nonempty interior, it is the closure of its interior, and so the inequality
$\langle v, z \rangle \le 0$ holds for all $v \in T_C(\infty_I).$ Thus, $z \in N_C(\infty_I).$
\end{proof}

\begin{remark}\label{Remark38}{\rm
The interior of $T_C(\infty_I)$ is nonempty if and only if  $N_C(\infty_I)$ is a pointed cone in the sense that $v \in  N_C(\infty_I) \setminus \{0\}$ implies $-v \not \in N_C(\infty_I)$ (see \cite[Exercise~6.22]{Rockafellar1998}).
}\end{remark}

An analytical characterization of Clarke's tangent cones is as follows.
 
\begin{proposition}
A vector $v \in \mathbb{R}^n$ belongs to $T_C(\infty_I)$ if and only if the following equality holds:
\begin{eqnarray*}
\limsup_{\pi(x) \to \infty, \ d_C(x) \to 0, \ t \searrow 0} \frac{d_C(x + t v) - d_C(x)}{t} &=& 0.
\end{eqnarray*}
\end{proposition}

\begin{proof}
Let $\gamma(v)$ be the left-hand side of the above expression. Then $\gamma(v) \ge 0.$

Let $v \in T_C(\infty_I).$ Choose sequences $x_k \in \mathbb{R}^n$ and $t_k \in (0, +\infty)$ such that
\begin{equation*}
\pi(x_k) \to \infty, \quad t_k \searrow  0, \quad d_C(x_k) \to 0, \quad \gamma(v) = \lim_{k \to \infty} \frac{d_C(x_k + t_k v) - d_C(x_k)}{t_k}.
\end{equation*}
Let $\overline{x}_k$ in $C$ satisfy
\begin{eqnarray*}
\|x_k - \overline{x}_k\| &\le&  d_C(x_k) + t_k^2.
\end{eqnarray*}
Then $\pi(\overline{x}_k) \to \infty.$ By the definition of $T_C(\infty_I),$ there exists a sequence $v_k$ converging to $v$ such that $\overline{x}_k + t_k v_k \in C$ for all $k.$ Since the distance function $d_C(\cdot)$ is 1-Lipschitz, it follows that
\begin{eqnarray*}
d_C(x_k + t_k v) 
&\le& d_C (\overline{x}_k + t_k v_k)  + \|x_k - \overline{x}_k\| + t_k \|v - v_k\| \\
&=& \|x_k - \overline{x}_k\| + t_k \|v - v_k\| \\
&\le& d_C(x_k) + t_k^2 + t_k \|v - v_k\|.
\end{eqnarray*}
Therefore, 
\begin{eqnarray*}
\gamma(v) \ = \ \lim_{k \to \infty} \frac{d_C(x_k + t_k v) - d_C(x_k)}{t_k} &\le& 0,
\end{eqnarray*}
and so $\gamma(v) = 0.$

Conversely, assume that $\gamma(v) = 0$ and that sequences $x_k \in C$ with $\pi(x_k) \to \infty$ and $t_k \searrow 0$ are given.
Then
\begin{eqnarray*}
\lim_{k \to \infty} \frac{d_C(x_k + t_k v)}{t_k} \ = \  \lim_{k \to \infty} \frac{d_C(x_k + t_k v) - d_C(x_k)}{t_k} &=& 0.
\end{eqnarray*}
Choose any $\widetilde{x}_k \in C$ so that
\begin{eqnarray*}
\|\widetilde{x}_k - (x_k + t_k v) \| &\le& d_C(x_k + t_k v) + t_k^2
\end{eqnarray*}
and let
\begin{eqnarray*}
v_k := \frac{\widetilde{x}_k - x_k}{t_k}.
\end{eqnarray*}
Then $x_k + t_k v = \widetilde{x}_k \in C.$ Furthermore, we have
\begin{eqnarray*}
\|v_k  - v \| &=& \frac{\|\widetilde{x}_k - (x_k + t_k v)\|}{t_k} \ \le \  \frac{d_C(x_k + t_k v)}{t_k} + t_k,
\end{eqnarray*}
which yields $v_k \to v.$ By definition, therefore, $v \in T_C(\infty_I).$
\end{proof}

\section{Subgradients at infinity and optimality conditions at infinity} \label{Section4}

In this section, let $f \colon \mathbb{R}^n \to \overline{\mathbb{R}} := \mathbb{R} \cup \{\infty\}$ be an extended real valued function. 
Fix the index set $I := \{1, \ldots, n\} \subset \{1, \ldots, n, n + 1\}$ and consider the projection $\pi \colon \mathbb{R}^n \times \mathbb{R} \to \mathbb{R}^{n}, (x, y) \mapsto x.$ By definition, then $\pi(\mathrm{epi} f) = \{x \in \mathbb{R}^n \mid  f(x) < \infty\}.$

\subsection{Subgradients at infinity}

\begin{definition}{\rm
We define the set of {\em subgradients} (or {\em generalized gradients})  {\em of $f$ at infinity} by
\begin{eqnarray*}
\partial f(\infty) &:=& \{\xi \in \mathbb{R}^n \ | \ (\xi, -1) \in N_{\textrm{epi} f}(\infty_I)\}.
\end{eqnarray*}
}\end{definition}

Observe that $\partial f(\infty)$ is a (possibly empty) closed convex set in $\mathbb{R}^n$ because the cone $N_{\textrm{epi} f}(\infty_I)$ is closed  convex. The following simple examples illustrate our definition. 

\begin{example}{\rm
(i) Let $n := 1$ and consider the function $f \colon \mathbb{R} \to \mathbb{R}$ defined by
$$f(x) := 
\begin{cases}
0 & \textrm{ if } x \le 0, \\
-x & \textrm{ otherwise.}
\end{cases}$$
A direct calculation shows that
\begin{eqnarray*}
T_{\mathrm{epi} f}(\infty_{I}) &=& \{(v_1, v_2) \in \mathbb{R}^2 \ | \ v_2 \ge 0 \textrm{ and }  v_1 + v_2 \ge 0\}, \\
N_{\mathrm{epi} f}(\infty_{I}) &=& \{(w_1, w_2) \in \mathbb{R}^2 \ | \ w_1 \le 0, w_2 \le 0, \textrm{ and } w_2 - w_1 \le 0 \}, \\
\partial f(\infty) &=& [-1, 0].
\end{eqnarray*}

\noindent
(ii) Consider the function $f \colon \mathbb{R} \to \mathbb{R}, x \mapsto e^x.$ We have 
\begin{eqnarray*}
T_{\mathrm{epi} f}(\infty_{I}) &=& \{(v_1, v_2) \in \mathbb{R}^2 \ | \ v_1 \le 0 \textrm{ and }  v_2 \ge 0\}, \\
N_{\mathrm{epi} f}(\infty_{I}) &=& \{(w_1, w_2) \in \mathbb{R}^2 \ | \ w_1 \ge 0 \textrm{ and }  w_2 \le 0\}, \\
\partial f(\infty) &=&  [0, +\infty).
\end{eqnarray*}

\noindent
(iii) Consider the function $f \colon \mathbb{R} \to \mathbb{R}, x \mapsto x^3.$ We have 
\begin{eqnarray*}
T_{\mathrm{epi} f}(\infty_{I}) &=& \{(v_1, v_2) \in \mathbb{R}^2 \ | \ v_1 \le 0\}, \\
N_{\mathrm{epi} f}(\infty_{I}) &=& \{(w_1, w_2) \in \mathbb{R}^2 \ | \ w_1 \ge 0 \textrm{ and }  w_2 = 0\}, \\
\partial f(\infty) &=& \emptyset.
\end{eqnarray*}
}\end{example}

\begin{corollary} \label{Corollary43}
Let $f$ be lower semi-continuous. Suppose $\partial f(\infty) \neq \emptyset$ and it does not contain any entire line. Then the multifunction $\partial f \colon \mathbb{R}^n \cup \{\infty\} \rightrightarrows \mathbb{R}^n$ is closed at infinity, in the sense that 
\begin{eqnarray*}
x_k \to \infty, \ \xi_k \in \partial f(x_k), \ \textrm { and  } \ \xi_k \to \xi \quad \implies \quad \xi \in \partial f(\infty). 
\end{eqnarray*}
\end{corollary}

\begin{proof}
The set $\mathrm{epi} f$ is closed because $f$ is lower semi-continuous. In view of Corollary~\ref{Corollary37} and Remark~\ref{Remark38}, it suffices to show that the cone $N_{\mathrm{epi} f}(\infty_I)$ is pointed. 
To see this, we first observe that $(0, r) \in T_{\mathrm{epi} f}(\infty_I)$ for all $r \ge 0,$ and so
$$ N_{\mathrm{epi} f}(\infty_I) \subset \{ (x, y)\in \mathbb{R}^n\times \mathbb{R}\mid y \le 0  \}.$$
Consequently, the cone $N_{\mathrm{epi} f}(\infty_I)$ fails to be pointed if and only if for some $u\neq 0$ it contains both $(u,0)$ and $(-u,0)$. For an arbitrary element of $N_{\mathrm{epi} f}(\infty_I)$ of the form $(\xi,-1)$, this property of $u$ is equivalent to having
$$(\xi,-1) + t(u,0) \in N_{\mathrm{epi} f}(\infty_I) \quad \textrm{ for all } \quad t \in \mathbb{R}.$$
Thus, if $\partial f(\infty) $ contains an element $\xi,$ the property of $u$ is equivalent to the line $\{ \xi + tu \mid t \in \mathbb{R}\}$ being contained in $\partial f(\infty).$
\end{proof}

\begin{proposition}[A version at infinity of the Fermat theorem]  \label{Proposition44}
Let $f$ be bounded below. If there exists a sequence $x_k \to \infty$ such that $f(x_k) \to f_* := \inf_{x \in \mathbb{R}^n} f(x),$ then $0 \in \partial f(\infty).$
\end{proposition}
\begin{proof}
We will show that $(0, -1) \in N_{\mathrm{epi} f}(\infty_I),$ or equivalently, that 
\begin{eqnarray*} 
r \ge 0 \quad \textrm{ for all } \quad (v, r) \in T_{\mathrm{epi} f}(\infty_I).
\end{eqnarray*}
To this end, take any $(v, r) \in T_{\mathrm{epi} f}(\infty_I).$ Let
$$t_k := 
\begin{cases}
\sqrt{f(x_k) - f_*} & \textrm{ if } f(x_k) > f_*, \\
\frac{1}{k} & \textrm{ otherwise.}
\end{cases}$$
Then $t_k > 0$ and $t_k \to 0.$ Passing to a subsequence if necessary, we can suppose that the sequence $t_k$ is decreasing. By the definition of $T_{\mathrm{epi} f}(\infty_I),$ there exists a sequence $(v_k, r_k)$ converging to $(v, r)$ such that
$(x_k, f(x_k)) + t_k (v_k, r_k) \in \mathrm{epi} f.$ But then
\begin{eqnarray*}
f(x_k) + t_k r_k &\ge& f(x_k + t_k v_k) \ \ge \ f_*.
\end{eqnarray*}
It follows that either $r_k \ge - \sqrt{f(x_k) - f_*}$ or $r_k \ge 0.$ Consequently, $r = \lim_{k \to \infty} r_k \ge 0,$ which completes the proof.
\end{proof}

\subsection{Subderivatives at infinity}

The set $\partial f(\infty)$ of subgradients at infinity corresponds to a kind of directional derivative at infinity that for general functions has a complicated
description but can be reduced to simpler formulas in many important cases (see Proposition~\ref{Proposition412} below).

The {\em (upper) subderivative of $f$ at infinity with respect to a vector $v \in \mathbb{R}^n$} is defined by
\begin{eqnarray*}
f^{\uparrow}(\infty; v) &:=& \lim_{\epsilon \searrow 0} \limsup_{x \xrightarrow{\mathrm{dom} f} \infty, \ t \searrow 0} 
\inf_{\ w \in \B_\epsilon(v)} \frac{f(x + t w) - f(x)}{t},
\end{eqnarray*}
where the notation $x \xrightarrow{\mathrm{dom} f} \infty$ means that $x \to \infty$ with $x \in \mathrm{dom} f.$ Observe that $f^{\uparrow}(\infty; 0) \le 0.$ Furthermore, we have the following relationship between Clarke's tangent cones and subderivatives at infinity.

\begin{proposition} \label{Proposition45}
The tangent cone $T_{\mathrm{epi} f}(\infty_I)$ is the epigraph of the function $f^{\uparrow}(\infty; \cdot);$ that is, a point $(v, r)$ of $\mathbb{R}^n \times \mathbb{R}$ belongs to $T_{\mathrm{epi} f}(\infty_I)$ if and only if $f^{\uparrow}(\infty; v) \le r.$
\end{proposition}

\begin{proof}
{\em Necessity.} Let $(v, r) \in T_{\mathrm{epi} f}(\infty_I)$ and let any $\epsilon > 0$ be given. It suffices to show 
\begin{eqnarray} \label{Clarke1}
\limsup_{x \xrightarrow{\mathrm{dom} f} \infty, \ t \searrow 0}  \inf_{w \in \B_\epsilon(v)} \frac{f(x + t w) - f(x)}{t} &\le& r.
\end{eqnarray}
To see this, let $x_k \xrightarrow{\mathrm{dom} f} \infty$ and $t_k \searrow  0.$ Since $(v, r) \in T_{\mathrm{epi} f}(\infty_I),$  there exists a sequence $(v_k, r_k)$ converging to $(v, r)$ such that $(x_k, f(x_k)) + t_k (v_k, r_k) \in \mathrm{epi} f$ for all $k.$ Then
\begin{eqnarray*}
f(x_k + t_k v_k) &\le& f(x_k) + t_k r_k.
\end{eqnarray*}
Consequently, for all $k$ sufficiently large,
\begin{eqnarray*}
\inf_{w \in \B_\epsilon(v)} \frac{f(x_k + t_k w) - f(x_k)}{t_k} &\le& \frac{f(x_k + t_k v_k) - f(x_k)}{t_k} \ \le \ r_k,
\end{eqnarray*}
and the inequality~\eqref{Clarke1} follows.

{\em Sufficiency.} Suppose $f^{\uparrow}(\infty; v) \le r$ and let $(x_k, y_k)$ be any sequence in $\mathrm{epi} f$ with $x_k \to\infty$ and $t_k$ any sequence decreasing to $0.$ In view of Lemma~\ref{Lemma33}, in order to show that $(v, r)$ belongs to $T_{\mathrm{epi} f}(\infty_I)$ it suffices to construct a sequence $(v_k, r_k)$ converging to $(v, r)$ such that
\begin{eqnarray*}
(x_k, y_k) + t_k(v_k, r_k) &\in& \mathrm{epi} f \quad \textrm{ infinitely time,}
\end{eqnarray*}
that is,
\begin{eqnarray} \label{Clarke2}
f(x_k + t_k v_k) &\le& y_k + t_k r_k \quad \textrm{ infinitely time.}
\end{eqnarray}

Since $f^{\uparrow}(\infty; v) \le r,$ for each positive integer $\ell,$ there is a real number $\epsilon := \epsilon({\ell}) \in (0, \frac{1}{\ell})$ such that
\begin{eqnarray*}
\limsup_{x \xrightarrow{\mathrm{dom} f} \infty, \ t \searrow 0} 
\inf_{\ w \in \B_{\epsilon}(v)} \frac{f(x + t w) - f(x)}{t} &\le& r + \frac{1}{\ell}.
\end{eqnarray*}
Since $x_k \to \infty$ with $f(x_k) \le y_k < \infty$ and $t_k \searrow 0,$ there exists an index $k = k(\ell) > \ell$ such that
\begin{eqnarray*}
\inf_{w \in \B_{\epsilon}(v)} \frac{f(x_k + t_k w) - f(x_k)}{t_k} &\le& r + \frac{2}{\ell},
\end{eqnarray*}
and therefore a point $v_k \in \B_{\epsilon}(v)$ such that
\begin{eqnarray*}
\frac{f(x_k + t_k v_k) - f(x_k)}{t_k} &\le& r + \frac{3}{\ell}.
\end{eqnarray*}
Let us define, for each index $k$ in the subsequence  $k(1), k(2), \ldots,$
\begin{eqnarray*}
r_k &:=& \max \left\{ r, \frac{f(x_k + t_k v_k) - f(x_k)}{t_k} \right\}.
\end{eqnarray*}
Clearly, $r_k$ converges to $r$ and the inequality in \eqref{Clarke2} is satisfied.
\end{proof}

\begin{corollary} \label{Corollary46}
The function $f^{\uparrow}(\infty; \cdot)$ is convex, positively homogeneous,  and lower semi-continuous.
\end{corollary}
\begin{proof}
By Corollary~\ref{Corollary35}, $T_{\mathrm{epi} f}(\infty_I)$ is a closed convex cone. This fact, together with Proposition~\ref{Proposition45}, implies the desired conclusion. 
\end{proof}

An analytical characterization of $\partial f(\infty)$ is given in the following corollary.

\begin{corollary} \label{Corollary47}
We have $\partial f(\infty) = \emptyset$ if and only if $f^{\uparrow}(\infty; 0) = -\infty.$ Otherwise, we have
\begin{eqnarray*}
\partial f(\infty) &=& \{\xi \in \mathbb{R}^n \ | \ \langle \xi, v \rangle \le f^{\uparrow}(\infty; v) \quad \textrm{ for all } \quad v \in \mathbb{R}^n\}, \\
f^{\uparrow}(\infty; v) &=& \sup\{ \langle \xi, v \rangle \ | \ \xi \in \partial f(\infty)\}, \\
\partial f^{\uparrow}(\infty; \cdot)(0) &=& \partial f(\infty). 
\end{eqnarray*}
\end{corollary}

\begin{proof}
By definition, $\xi \in \partial f(\infty)$ if and only if $(\xi, -1) \in N_{\mathrm{epi} f}(\infty_I).$ Equivalently, 
\begin{eqnarray*}
\langle (\xi, -1), (v, r) \rangle &\le& 0 \quad \textrm{ for all } \quad (v, r) \in T_{\mathrm{epi} f} (\infty_I).
\end{eqnarray*}
This, together with Proposition~\ref{Proposition45}, implies that $\xi \in \partial f(\infty)$ if and only if
\begin{eqnarray*}
\langle \xi, v \rangle &\le& r \quad \textrm{ for all } v \in \mathbb{R}^n \textrm{ and all } r \ge f^{\uparrow}(\infty; v).
\end{eqnarray*}
If $f^{\uparrow}(\infty; v) = -\infty$ for some $v,$ there can be no such $\xi,$ that is, $\partial f(\infty) = \emptyset.$ On the other hand, we have $f^{\uparrow}(\infty; \lambda v) = \lambda f^{\uparrow}(\infty; v)$ for all $v$ and $\lambda > 0$ because the function $f^{\uparrow}(\infty; \cdot)$ is positively homogeneous (we use the convention $\lambda(-\infty) = - \infty).$ Therefore, $f^{\uparrow}(\infty; \cdot) = -\infty$ somewhere if and only if it is $-\infty$ at $0.$

To complete the proof of the corollary we assume that $f^{\uparrow}(\infty; 0) > -\infty.$ Then $f^{\uparrow}(\infty; v)  > -\infty$ for all $v \in \mathbb{R}^n.$ Moreover, by definition, $f^{\uparrow}(\infty; 0) \leq 0.$ It follows easily that $f^{\uparrow}(\infty; 0) = 0.$ 
Since the function $f^{\uparrow}(\infty; \cdot)$ is convex (see Corollary~\ref{Corollary46}), its subgradient set (in the sense of convex analysis) at $v = 0$ is nonempty (see \cite[Theorem~23.4]{Rockafellar1970}) and agrees with $\partial f^{\uparrow}(\infty; \cdot)(0)$ (see \cite[Proposition~2.2.7]{Clarke1990} and \cite[Theorem~5]{Rockafellar1980}), i.e.,
\begin{eqnarray*}
\emptyset \ \ne \  \{\xi \in \mathbb{R}^n \ | \ \langle \xi, v \rangle \le f^{\uparrow}(\infty; v) \quad \textrm{ for all } \quad v \in \mathbb{R}^n\} &=& \partial f^{\uparrow}(\infty; \cdot)(0).
\end{eqnarray*}
Consequently, the set $\partial f(\infty)$ is nonempty because of the second assertion, as we have seen in the beginning of the proof. 
The last assertion follows immediately. Finally, by Corollary~\ref{Corollary46} and \cite[Corollary~13.2.1]{Rockafellar1970} (see also \cite[Theorem~8.24]{Rockafellar1998}), $f^{\uparrow}(\infty; \cdot)$ is the support function of $\partial f(\infty),$ which is the third assertion. The proof is complete.
\end{proof}

\begin{example}\label{Example48}{\rm
Let $C \subset \mathbb{R}^n$ be an unbounded set. Define the {\em indicator} function $\delta_{C} \colon \mathbb{R}^n \to \overline{\mathbb{R}}$ of the set $C$ by
$$\delta_{C}(x) :=
\begin{cases}
0 & \textrm{ if } x \in C, \\
+\infty & \textrm{ otherwise.}
\end{cases}$$
By definition, $\mathrm{epi} \delta_C = C \times [0, +\infty).$ Hence
$T_{\mathrm{epi} \delta_C} (\infty_I) = T_C(\infty_I) \times [0, +\infty)$ and $N_{\mathrm{epi} \delta_C} (\infty_I) = N_C(\infty_I) \times (-\infty, 0].$ Therefore, $\partial \delta_{C} (\infty) = N_{C}(\infty_{I})$ and
$$\delta_C^{\uparrow}(\infty; v) = 
\begin{cases}
0 & \textrm{ if } v \in T_C(\infty_I), \\
+\infty & \textrm{ otherwise.}
\end{cases}$$
}\end{example}

\subsection{Directionally Lipschitz functions at infinity}

We will say that $f$ is {\em directionally Lipschitz at infinity with respect to a vector} $v \in \mathbb{R}^n$ if the quantity 
\begin{eqnarray*}
f^{\dagger}(\infty; v) &:=& \limsup_{x \xrightarrow{\mathrm{dom} f} \infty, \ t \searrow 0, \ w \to v}  \frac{f(x + t w) - f(x)}{t}
\end{eqnarray*}
is not $+\infty.$ The function $f$ is said to be {\em directionally Lipschitz at infinity} if it is directionally Lipschitz at infinity with respect to at least one $v$ in $\mathbb{R}^n.$

\begin{proposition}\label{Proposition49}
Let $f$ be lower semi-continuous. Then, $f$ is directionally Lipschitz at infinity with respect to $v$ if and only if there exists a real number $r$ such that $(v, r) \in {\rm int}T_{\mathrm{epi} f} (\infty_I).$
\end{proposition}

\begin{proof}
{\em Sufficiency.} Let $(v, r) \in {\rm int}T_{\mathrm{epi} f} (\infty_I).$ Since $f$ is lower semi-continuous, ${\mathrm{epi} f}$ is a closed set in $\mathbb{R}^n \times \mathbb{R}.$ In light of Theorem~\ref{Theorem36}, there exist constants $R > 0$ and $\epsilon > 0$ such that 
\begin{eqnarray*}
({x}, {y}) + t(v', r') \in \mathrm{epi} f  \ \textrm{ for all } \ ({x}, {y}) \in \mathrm{epi} f \setminus \B_{R, I}, \ (v', r') \in \B_{\epsilon}(v, r), \ t \in (0,\epsilon). 
\end{eqnarray*}
Letting ${y} := f({x})$ and $r' := r,$ we get $f({x} + tv') \le f({x}) + t r,$ and so
\begin{eqnarray*}
\dfrac{f({x} + tv') - f({x})}{t} &\le& r.
\end{eqnarray*}
It follows that $f$ is directionally Lipschitz at infinity with respect to $v.$

{\em Necessity.} Assume that $f$ is directionally Lipschitz at infinity with respect to $v,$ i.e., $f^{\dagger}(\infty; v) < +\infty.$ Then we can find a real number $L > 0$ satisfying
\begin{eqnarray*}
\limsup_{x \xrightarrow{\mathrm{dom} f} \infty, \ t \searrow 0, \ w \to v}  \frac{f(x + t w) - f(x)}{t} &<& L.
\end{eqnarray*}
Consequently, there are constants $R > 0$ and $\epsilon_1 > 0$ such that for all ${x} \in \mathrm{dom} f \setminus \B_R$, $v'\in \B_{\epsilon_1}(v)$ and $t \in (0,\epsilon_1),$ we have
\begin{eqnarray*}
\dfrac{f({x} + tv') - f({x})}{t} &\le& L,
\end{eqnarray*}
and hence
\begin{eqnarray*}
f({x} + tv') &\le& f({x}) + t L \ \leq \ {y} + t r',
\end{eqnarray*}
where ${y} \ge f({x})$ and $r' \ge L.$ Letting $r := 2 L$ and $\epsilon := \min \{\epsilon_1, \frac{L}{2} \},$ we obtain
\begin{eqnarray*}
({x}, {y}) + t(v', r') \in \mathrm{epi} f  \ \textrm{ for all } \  ({x}, {y}) \in \mathrm{epi} f \setminus \B_{R, I}, \ (v', r') \in \B_{\epsilon}(v, r), \ t \in (0,\epsilon),
\end{eqnarray*}
which yields $(v, r) \in {\rm int}T_{\mathrm{epi} f} (\infty_I)$ by Theorem~\ref{Theorem36}.
\end{proof}

\begin{example}{\rm
Let $n := 2$ and consider the function $f \colon \mathbb{R}^2 \to \mathbb{R}, (x_1, x_2) \mapsto e^{x_1} + x_2.$ A simple calculation shows that
\begin{eqnarray*}
T_{\mathrm{epi} f}(\infty_{I}) &=& \{(v_1, v_2, r) \in \mathbb{R}^2 \times \mathbb{R} \ | \ v_1 \le 0, v_2 \le r\}.
\end{eqnarray*}
In view of Proposition~\ref{Proposition49}, $f$ is directionally Lipschitz at infinity with respect to $(v_1, v_2) \in \mathbb{R}^2$ if and only if $v_1 < 0.$
}\end{example}

To simplify notation in what follows, let
\begin{eqnarray*}
D^{\uparrow}_{f}(\infty) &:=& \{ v \in \mathbb{R}^n \mid f^{\uparrow}(\infty; v) < +\infty\}, \\
D^{\dagger}_{f}(\infty) &:=& \{ v \in \mathbb{R}^n \mid f^{\dagger}(\infty; v) < +\infty\}.
\end{eqnarray*}

\begin{corollary} \label{Corollary411}
Let $f$ be lower semi-continuous. Then for every $v \in \mathrm{int} D^{\uparrow}_{f}(\infty),$ the function $f$ is directionally Lipschitz at infinity with respect to $v.$
\end{corollary}

\begin{proof}
The set $D^{\uparrow}_{f}(\infty)$ is a convex cone as the function $f^{\uparrow}(\infty; \cdot)$ is convex and positively homogeneous (by Corollary~\ref{Corollary46}). Take arbitrary $v \in \mathrm{int} D^{\uparrow}_{f}(\infty)$ and let $r > f^{\uparrow}(\infty; v).$ In light of \cite[Theorem~6.8]{Rockafellar1970}, $(v, r) \in \mathrm{int}\, \mathrm{epi} f^{\uparrow}(\infty; \cdot).$ This, together with Propositions~\ref{Proposition45} and \ref{Proposition49}, implies that $f$ is directionally Lipschitz at infinity with respect to the vector $v.$
\end{proof}

One of the situations where a considerable simplification is possible in analyzing subderivatives at infinity is the case where there exists at least one $v \in \mathbb{R}^n$ such that $f^{\dagger}(\infty; v) < +\infty,$  i.e., when the set $D^{\dagger}_{f}(\infty)$ is nonempty.

\begin{proposition} \label{Proposition412}
Let $f$ be lower semi-continuous such that $D^{\dagger}_{f}(\infty) \neq \emptyset.$ Then
\begin{eqnarray*}
D^{\dagger}_{f}(\infty) &=& {\rm int} D^{\uparrow}_{f}(\infty).
\end{eqnarray*}
Furthermore, $f^{\uparrow}(\infty; \cdot)$ is continuous at each $v \in 	D^{\dagger}_{f}(\infty) $ and agrees there with $f^{\dagger}(\infty; \cdot)$. 
\end{proposition}

\begin{proof}
By Proposition~\ref{Proposition49}, we have
\begin{eqnarray*}
D^{\dagger}_{f}(\infty) &=& \pi({\rm int}T_{\mathrm{epi} f} (\infty_I)).
\end{eqnarray*}
(Recall that $\pi \colon \mathbb{R}^n \times \mathbb{R} \to \mathbb{R}^n, (v, r) \mapsto v,$ denotes the projection on the first component.) Hence, in order to prove the first assertion of the theorem, it suffices to establish
\begin{eqnarray}\label{3.8}
\pi(\mathrm{int}T_{\mathrm{epi} f} (\infty_I)) &=& {\rm int}\{ v \in \mathbb{R}^n \mid f^{\uparrow}(\infty; v) < +\infty \}.
\end{eqnarray}

Let $v \in \pi(\mathrm{int}T_{\mathrm{epi} f} (\infty_I)).$ Then $(v, r) \in \mathrm{int}T_{\mathrm{epi} f} (\infty_I)$ for some $r \in \mathbb{R}.$ In view of Proposition~\ref{Proposition45}, $(v, r) \in \mathrm{int}\, \mathrm{epi} f^{\uparrow}(\infty; \cdot).$ Consequently, $f^{\uparrow}(\infty; w) < +\infty$ for all $w$ near $v,$ and so $v$ lies in the right-hand side of \eqref{3.8}.

Conversely, let $v$ be in the right-hand side of~\eqref{3.8}. Then, there are real numbers $\epsilon > 0$ and $r$ such that $f^{\uparrow}(\infty; v') < r - 1$ for all $v' \in \B_{\epsilon }(v).$ It follows that $(v', r') \in {\rm epi} f^{\uparrow}(\infty; \cdot) = T_{\mathrm{epi} f}(\infty_I)$ for all $(v', r')$ close to $(v, r),$ equivalently, $(v, r) \in \mathrm{int} T_{\mathrm{epi} f}(\infty_I),$ and so $v \in \pi({\rm int}T_{{\rm epi} f}(\infty_{I})).$ Therefore, the equation~\eqref{3.8} holds, which proves the first assertion of the theorem.

Since the function $f^{\uparrow}(\infty; \cdot)$ is convex, it is continuous on the interior of the set $D^{\uparrow}_f(\infty),$ provided it is bounded above in a neighborhood of one point (see \cite[Theorem~10.1]{Rockafellar1970}); this is the case precisely because $D^{\dagger}_{f}(\infty) \neq \emptyset.$ The second assertion of the theorem is proved.

The last assertion of the theorem follows directly from the following expressions:
\begin{eqnarray*}
f^{\uparrow}(\infty; v) &=&{\rm inf }\{ r \in \mathbb{R} \mid (v, r) \in T_{\mathrm{epi} f} (\infty_I) \},\\
f^{\dagger}(\infty; v) &=&{\rm inf }\{ r \in \mathbb{R} \mid (v, r) \in {\rm int}T_{\mathrm{epi} f} (\infty_I) \},
\end{eqnarray*}
where $v$ belongs to $D^{\dagger}_f(\infty).$ 
\end{proof}

We next give a formula for the subgradients at infinity of a sum of two functions.
\begin{proposition}[sum rule] \label{Proposition413} 
Let $f_1$ and $f_2$ be lower semi-continuous such that $D^{\uparrow}_{f_1}(\infty) \cap \mathrm{int} D^{\uparrow}_{f_2}(\infty) \ne \emptyset.$ Then
\begin{eqnarray}
(f_1 + f_2)^{\uparrow}(\infty; v) &\le& f_1^{\uparrow}(\infty; v) + f_2^{\uparrow}(\infty; v) \quad \textrm{ for all } \quad v \in \mathbb{R}^n, \label{3.22} \\
\partial (f_1 + f_2)(\infty) &\subset& \partial f_1(\infty) +\partial f_2(\infty).  \label{sumrule}
\end{eqnarray}
\end{proposition}

\begin{proof}
For simplicity of notation, we let $f_0 := f_1 + f_2$ and $\ell_i(v) := f_i^{\uparrow}(\infty; v)$ for $i = 0, 1, 2.$
Our assumption implies that $\mathrm{int} D^{\uparrow}_{f_2}(\infty) \ne \emptyset.$ By Corollary~\ref{Corollary411}, the function $f_2$ is directionally Lipschitz at infinity. Take any $v \in D^{\uparrow}_{f_1}(\infty)\cap {\rm int} D^{\uparrow}_{f_2}(\infty)$ and let $r > \ell_2(v) = f_2^{\uparrow}(\infty; v).$ In view of Proposition~\ref{Proposition412}, $f_2^{\dagger}(\infty; v) = f_2^{\uparrow}(\infty; v) < r.$ Hence, there exist $R>0$ and $\epsilon>0$ such that for all $x \in \mathrm{dom} f_2 \setminus \B_R$, $w \in \B_{\epsilon}(v)$ and $t \in (0,\epsilon),$ we have
\begin{eqnarray*}
\dfrac{f_2(x + tw) - f_2(x)}{t} &\le& r,
\end{eqnarray*}
and so
\begin{eqnarray*}
\frac{f_0(x + t w) - f_0(x)}{t} 
&=& \frac{f_1(x + t w) - f_1(x)}{t} +\frac{f_2(x + t w) - f_2(x)}{t} \\
&\le& \frac{f_1(x + t w) - f_1(x)}{t} +r. \nonumber
\end{eqnarray*}
Note that $\mathrm{dom} f_0 \subset \mathrm{dom} f_1 \cap \mathrm{dom} f_2.$ Therefore,
\begin{eqnarray*}
\ell_0(v)  &=& \lim_{\epsilon \searrow 0} \limsup_{x \xrightarrow{\mathrm{dom} f_0} \infty, \ t \searrow 0}  \inf_{w \in \B_\epsilon(v)} \frac{f_0 (x + t w) - f_0(x)}{t} \\
&\le&  \lim_{\epsilon \searrow 0} \limsup_{x \xrightarrow{\mathrm{dom} f_1} \infty, \ t \searrow 0}  \inf_{w \in \B_\epsilon(v)} \left\{\frac{f_1(x + t w) - f_1(x)}{t} + r\right\} \\
&=&  \lim_{\epsilon \searrow 0} \limsup_{x \xrightarrow{\mathrm{dom} f_1} \infty, \ t \searrow 0}  \inf_{w \in \B_\epsilon(v)} \left\{\frac{f_1(x + t w) - f_1(x)}{t}\right\}  + r \\
&=&  \ell_1(v) + r.
\end{eqnarray*}
Since this inequality is true for all $r > \ell_2(v) = f_2^{\uparrow}(\infty; v),$ we derive the inequality~\eqref{3.22}, which we still need to establish for general $v.$
If either $\ell_1(v) = +\infty$ or $\ell_2(v) = +\infty$ then \eqref{3.22} is trivial, so assume that $v$ belongs to the set  $D^{\uparrow}_{f_1}(\infty)\cap D^{\uparrow}_{f_2}(\infty).$ By assumption, there exists a vector $\overline{v}$ in $D^{\uparrow}_{f_1}(\infty)\cap {\rm int} D^{\uparrow}_{f_2}(\infty).$ Since 
the sets $D^{\uparrow}_{f_1}(\infty)$ and $ D^{\uparrow}_{f_2}(\infty)$ are convex, $v_{t} := (1 - t)v + t \overline{v} \in D^{\uparrow}_{f_1}(\infty)\cap {\rm int} D^{\uparrow}_{f_2}(\infty)$ for all $t \in (0, 1).$ By the case already treated, then
\begin{eqnarray*}
\ell_0({v}_{t})  &\le& \ell_1({v}_{t}) + \ell_2({v}_{t}).
\end{eqnarray*}
The functions $\ell_i$ are convex and lower semi-continuous, so $\lim_{t \to 0} \ell_i({v}_{t}) = \ell_i({v})$  (see \cite[Corollary~7.5.1]{Rockafellar1970}). Therefore, the inequality~\eqref{3.22} holds.

We now prove the inclusion~\eqref{sumrule}. There are two cases to be considered: 
\subsubsection*{Case 1: $\ell_1(0) = -\infty$ or $\ell_2(0) = -\infty$} 
Note that $\ell_i(0) = f_i^{\uparrow}(\infty; 0) \le 0$ (by definition). Hence the inequality~\eqref{3.22} implies $\ell_0(0) = -\infty,$ and so $\partial f_0(\infty) = \emptyset$ by Corollary~\ref{Corollary47}. Then the inclusion~\eqref{sumrule} is trivial. 

\subsubsection*{Case 2: $\ell_1(0) \neq -\infty$ and $\ell_2(0) \neq -\infty$} 
Assume that $\partial f_0(\infty) \ne \emptyset$ (otherwise the inclusion~\eqref{sumrule} is trivial). Then $\ell_0(0) > -\infty$ in view of Corollary~\ref{Corollary47}. Moreover, we have 
\begin{eqnarray}
\partial f_0 (\infty) 
& = & \{ \xi \in \mathbb{R}^n \mid \langle \xi, v \rangle \le \ell_0(v) \ \textrm{ for all } \ v \in \mathbb{R}^n \} \nonumber\\
& \subset & \{ \xi \in \mathbb{R}^n \mid \langle \xi, v \rangle \le (\ell_1 + \ell_2)(v) \ \textrm{ for all }  \ v \in \mathbb{R}^n \} \nonumber \\
& = & \partial(\ell_1 + \ell_2) (0), \label{PT26}
\end{eqnarray}
where the last equality is valid because $\ell_1 + \ell_2$ is a convex function satisfying $(\ell_1 + \ell_2)(0) = 0.$ The directionally Lipschitz property at infinity of $f_2$ implies that the function $f_2^{\uparrow}(\infty; \cdot)$ is continuous on the interior of the set $D_{f_2}^{\uparrow}(\infty)$ (see Proposition~\ref{Proposition412}). Our assumption thus provides the existence of a vector $v$ such that $\ell_1(v) < \infty$ and $\ell_2$ is bounded above on a neighborhood of $v.$ In view of \cite[Theorem~23.8]{Rockafellar1970}, we get
\begin{eqnarray*}
\partial(\ell_1 + \ell_2) (0) & = & \partial \ell_1(0) + \partial \ell_2(0) \ = \  \partial f_1(\infty) +\partial f_2(\infty),
\end{eqnarray*}
where the second equality holds true because $\ell_1$ and $\ell_2$ are convex functions satisfying $\ell_1(0) = \ell_2(0) = 0.$ Now 
the desired inclusion~\eqref{sumrule} follows directly from \eqref{PT26}.
\end{proof}

\subsection{Optimality conditions at infinity}

We are now in a position to prove the second of our main theorems, which provides necessary optimality conditions at infinity for optimization problems.

\begin{theorem} \label{Theorem414}
Let $f$ be lower semi-continuous and $C \subset \mathbb{R}^n$ be an unbounded closed set. 
Assume that one of the following conditions holds:
\begin{enumerate}[{\rm (i)}]
\item $D^{\uparrow}_{f}(\infty) \cap \mathrm{int}T_{C}(\infty_{I}) \ne \emptyset.$
\item $\mathrm{int}D^{\uparrow}_{f}(\infty) \cap T_{C}(\infty_{I}) \ne \emptyset.$
\end{enumerate}
If there exists a sequence $x_k \in C$ such that $x_k \to \infty$ and $f(x_k) \to \inf_{x \in C} f(x) > -\infty,$ then the following optimality conditions hold:
\begin{eqnarray*}
0 &\in& \partial f(\infty)+ N_{C}(\infty_{I}), \\
f^{\uparrow}(\infty; v) &\ge& 0 \quad \textrm{ for all } \quad v \in T_C(\infty_I).
\end{eqnarray*}
\end{theorem}

\begin{proof} 
Recall that $\delta_{C} \colon \mathbb{R}^n \to \overline{\mathbb{R}}$ denotes the indicator function of the set $C,$ which is 
lower semi-continuous (because $C$ is closed). By definition, then
\begin{eqnarray*}
\lim_{k \to \infty} (f + \delta_C)(x_k) &=& \inf_{x \in \mathbb{R}^n} \left(f + \delta_{C}\right)(x).
\end{eqnarray*}
Hence $0 \in \partial \left(f + \delta_{C}\right) (\infty)$ in light of Proposition~\ref{Proposition44}. Furthermore, we know from Example~\ref{Example48} that
$$D^{\uparrow}_{\delta_C}(\infty) = T_{C}(\infty_{I}) \quad \textrm{ and } \quad  \partial \delta_{C} (\infty) = N_{C}(\infty_{I}).$$
Therefore, either (i) or (ii) is sufficient for Proposition~\ref{Proposition413} to be applicable and yield the inclusion
\begin{eqnarray*}
0 &\in& \partial \left(f + \delta_{C}\right) (\infty) \ \subset \ \partial f (\infty) + \partial \delta_{C} (\infty) \ = \ \partial f (\infty) +  N_{C}(\infty_{I}),
\end{eqnarray*}
which gives the first statement. Moreover, by Corollary~\ref{Corollary47}, we have
\begin{eqnarray*}
(f + \delta_C)^{\uparrow}(\infty; v) &\ge& 0 \quad \textrm{ for all } \quad v \in \mathbb{R}^n.
\end{eqnarray*}
It follows from Proposition~\ref{Proposition413} that
\begin{eqnarray*}
f^{\uparrow}(\infty; v)  + \delta_C^{\uparrow}(\infty; v) &\ge& 0 \quad \textrm{ for all } \quad v \in \mathbb{R}^n.
\end{eqnarray*}
Combining this inequality with Example~\ref{Example48}, we obtain the second statement.
\end{proof}

\begin{example}{\rm 
Let $n := 2,$ $f(x_1, x_2) := x_1,$ and $C := \{(x_1, x_2) \in \mathbb{R}^2 \ | \ x_1 \ge 0, x_2 \ge 0, x_1x_2 \ge 1\}.$ Clearly $\inf_{(x_1, x_2) \in C} f(x_1, x_2) = 0$ and $f$ does not attain its infimum on $C.$ A direct calculation shows that $\partial f(\infty) = \{(1, 0)\}$ and $N_C(\infty_I) = (-\infty, 0] \times (-\infty, 0],$ and so the necessary optimality condition at infinity $(0, 0) \in \partial f(\infty) + N_C(\infty_I)$ holds.
}\end{example}

\section{Lipschitz functions at infinity} \label{Section5}

In this section, we will present some characterizations of the Lipschitz continuity at infinity for lower semi-continuous functions.
So let $f \colon \mathbb{R}^n \to {\mathbb{R}}$ be a real valued function. As before, we fix $I := \{1, \ldots, n\} \subset \{1, \ldots, n, n + 1\}$ and consider the projection $\pi \colon \mathbb{R}^n \times \mathbb{R} \to \mathbb{R}^{n}, (x, y) \mapsto x.$

\begin{definition}{\rm
We say that $f$ is {\em Lipschitz at infinity} if there exist constants $L > 0$ and $R > 0$ such that 
\begin{eqnarray*}
|f(x) - f(x')| &\le& L \|x - x'\| \quad \textrm{ for all } \quad x, x' \in \mathbb{R}^n \setminus \B_R.
\end{eqnarray*}
}\end{definition}

The {\em generalized Clarke derivative of $f$ at infinity with respect to} $v \in \mathbb{R}^n$ is defined by
\begin{eqnarray*}
f^0(\infty; v) &:=& \limsup_{x \to \infty, \ t \searrow 0} \frac{f(x + t v) - f(x)}{t}.
\end{eqnarray*}
By definition, if $f$ is Lipschitz at infinity, then $f^0(\infty; v)$ is finite for all $v \in \mathbb{R}^n.$ Furthermore, we have the following.

\begin{proposition} \label{Proposition52}
Let $f$ be Lipschitz at infinity. Then the epigraph of $f^0(\infty; \cdot)$ is $T_{\mathrm{epi} f}(\infty_I);$ that is, $(v, r) \in T_{\mathrm{epi} f}(\infty_I)$ if and only if $f^0(\infty; v) \le r.$
\end{proposition}

\begin{proof}
Let $(v, r) \in T_{\mathrm{epi} f}(\infty_I).$ By the definition of $f^0(\infty; v),$ there exist sequences $x_k \to \infty$ and $t_k \searrow 0$ such that 
\begin{eqnarray*}
f^0(\infty; v) &=& \lim_{k \to \infty} \frac{f(x_k + t_k v) - f(x_k)}{t_k}.
\end{eqnarray*}
By the definition of $T_{\mathrm{epi} f}(\infty_I),$ there exists a sequence $(v_k, r_k) \in \mathbb{R}^n \times \mathbb{R}$ tending to $(v, r)$ such that $(x_k, f(x_k)) + t_k(v_k, r_k) \in \mathrm{epi} f$ for all $k.$ Hence
\begin{eqnarray*}
f(x_k + t_k v_k) &\le& f(x_k) + t_k r_k.
\end{eqnarray*}
Equivalently,
\begin{eqnarray*}
\frac{f(x_k + t_k v_k) - f(x_k)}{t_k} &\le& r_k.
\end{eqnarray*}
Consequently,
\begin{eqnarray*}
\frac{f(x_k + t_k v) - f(x_k)}{t_k} &\le& 
\frac{f(x_k + t_k v_k) - f(x_k)}{t_k}  + L\|v_k - v\| \ \le \ r_k + L\|v_k - v\| .
\end{eqnarray*}
Letting $k \to \infty,$ we get $f^0(\infty; v) \le r.$

Conversely, take any $v \in \mathbb{R}^n$ and $\delta \ge 0.$ We will show that $(v, f^0(\infty; v) + \delta) \in T_{\mathrm{epi} f}(\infty_I).$
To see this, let $(x_k, y_k)$ be a sequence in $\mathrm{epi} f$ with $x_k \to \infty$ and let $t_k \searrow 0.$ Define
\begin{eqnarray*}
r_k &:=& \max \left\{f^0(\infty; v) + \delta, \frac{f(x_k + t_k v) - f(x_k)}{t_k} \right\}.
\end{eqnarray*}
Observe that $r_k$ tends to $f^0(\infty; v) + \delta$ because we know that
\begin{eqnarray*}
\limsup_{k \to \infty} \frac{f(x_k + t_k v) - f(x_k)}{t_k} &\le& f^0(\infty; v).
\end{eqnarray*}
Furthermore, we have
 \begin{eqnarray*}
y_k + t_k r_k 
&\ge& y_k + \big( f(x_k + t_k v) - f(x_k) \big) \\
&\ge& f(x_k) + \big( f(x_k + t_k v) - f(x_k) \big) \ = \ f(x_k + t_k v).
\end{eqnarray*}
Therefore, $(x_k, y_k) + t_k (v_k, r_k) \in \mathrm{epi} f,$ where $v_k := v,$ and so $(v, f^0(\infty; v) + \delta) \in T_{\mathrm{epi} f}(\infty_I).$
\end{proof}

We denote by $\Omega_f \subset \mathbb{R}^n$ the set of points where $f$ is not differentiable. By Rademacher's theorem (see, for example, \cite{Rockafellar1998}), if $f$ is Lipschitz at infinity, then there exists a real number $R > 0$ such that the set $\Omega_f \setminus \B_R$ has measure zero in $\mathbb{R}^n.$ 

\begin{proposition} \label{Proposition53}
Let $f$ be Lipschitz at infinity. Then 
\begin{eqnarray*}
\partial f(\infty) &=& \mathrm{co} \{ \lim \nabla f(x) \ | \ x \to \infty \textrm{ and } x\not \in \Omega_f\}.
\end{eqnarray*}
\end{proposition}

\begin{proof}
Let 
\begin{eqnarray*}
A &:=& \{ \lim \nabla f(x) \ | \ x \to \infty \textrm{ and } x\not \in \Omega_f\}.
\end{eqnarray*}
We first show that $\partial f(\infty)  \supset \mathrm{co} A.$ Since $\partial f(\infty) $ is a convex set, it suffices to prove $\partial f(\infty) \supset A.$
To see this, take any $\xi \in A.$ By definition, there exists a sequence $x_k \to \infty$ with $x_k \not \in \Omega_f$ such that $\nabla f(x_k) \to \xi.$
We will show that $\xi \in \partial f(\infty).$ In view of Proposition~\ref{Proposition52}, it suffices to show that $\langle \xi, v \rangle \le f^0(\infty; v)$ for all $v \in \mathbb{R}^n.$ To do this, take any $v \in \mathbb{R}^n.$ By \cite[Propositions~2.1.5~and~2.2.2]{Clarke1990}, we know that
\begin{eqnarray*}
\langle \nabla f(x_k), v \rangle & \le & f^0(x_k; v) \ := \ \limsup_{x' \to x_k, \ t \searrow 0} \frac{f(x' + t v) - f(x')}{t}.
\end{eqnarray*}
On the other hand, by definition, there exist $\overline{x}_k \in \mathbb{R}^n$ and $t_k \in \mathbb{R}$ with $\|x_k - \overline{x}_k\| \le \frac{1}{k}$ and $0 < t_k < \frac{1}{k}$ such that
\begin{eqnarray*}
f^0(x_k; v) &\le & \frac{f(\overline{x}_k + t_k v)  - f(\overline{x}_k)}{t_k} + \frac{1}{k}.
\end{eqnarray*}
Observe that $\overline{x}_k \to \infty$ and $t_k \to 0.$ Moreover, 
\begin{eqnarray*}
\langle \xi, v \rangle \ = \ \lim_{k \to \infty} \langle \nabla f(x_k), v \rangle 
& \le & \limsup_{k \to \infty} f^0(x_k; v)  \\
& \le & \limsup_{k \to \infty} \left[\frac{f(\overline{x}_k + t_k v)  - f(\overline{x}_k)}{t_k} + \frac{1}{k} \right ]\\
& \le & \limsup_{x \to \infty, \ t \searrow 0} \frac{f(x + t v)  - f(x)}{t} \ = \ f^0(\infty; v).
\end{eqnarray*}

We have proved that $\partial f(\infty)  \supset \mathrm{co} A.$ Since $\mathrm{co} A$ is a compact convex set in $\mathbb{R}^n,$ in view of Proposition~\ref{Proposition52}, in order to complete the proof it suffices to show that for all $v \in \mathbb{R}^n \setminus \{0\},$ we have
\begin{eqnarray*}
f^0(\infty; v) & \le & \limsup \{ \langle \nabla f(x), v \rangle \ | \ x \to \infty \textrm{ and } x\not \in \Omega_f\}.
\end{eqnarray*}
To see this, let $r$ be the right-hand side and take arbitrary $\epsilon > 0.$ By definition, there exists a constant $R > 0$ such that 
\begin{eqnarray*}
\langle \nabla f(x), v \rangle & \le & r + \epsilon \quad \textrm{ for all } \quad x \in \mathbb{R}^n \setminus (\B_R \cup \Omega_f).
\end{eqnarray*}
For each $x \in \mathbb{R}^n,$ we define $L_x := \{x + tv \ | \ 0 < t < \frac{R}{\|v\|}\}.$ Since $\Omega_f$ has measure zero in $\mathbb{R}^n,$ it follows from Fubini's theorem that for almost $x$ in $\mathbb{R}^n \setminus \B_R,$ the set $L_x \cap \Omega_f$ has measure zero in $L_x.$ Let $x$ be any point in $\mathbb{R}^n \setminus \B_{2R}$ having this property and let $t$ lie in $(0, \frac{R}{\|v\|}),$ and then
\begin{eqnarray*}
f(x + tv) - f(x) &  = & \int_0^t \langle \nabla f(x + sv) , v \rangle ds,
\end{eqnarray*}
because $\nabla f$ exists almost everywhere on $L_x.$ Since we have $\|x + sv \| \ge \|x\| - s\|v\| > R$ for $0 < s < t,$ it follows that
\begin{eqnarray*}
\langle \nabla f(x + sv) , v \rangle &\le& r + \epsilon,
\end{eqnarray*}
whence 
\begin{eqnarray*}
f(x + tv) - f(x) &\le & t (r + \epsilon).
\end{eqnarray*}
Since this inequality is true for all $x \in \mathbb{R}^n \setminus \B_{2R}$ except those in a set of measure zero and for all $t$ in $(0, \frac{R}{\|v\|})$ and since $f$ is continuous, it is in fact true for all such $x$ and $t.$ Therefore,
\begin{eqnarray*}
f^0(\infty; v) &\le& r + \epsilon.
\end{eqnarray*}
Since $\epsilon > 0$ is arbitrary, we get $f^0(\infty; v) \le r,$ which completes the proof.
\end{proof}

\begin{corollary}
If $f$ is Lipschitz at infinity, then $\partial (-f)(\infty) = -\partial f(\infty)$ and 
\begin{eqnarray*}
f^0(\infty; v) &  = & \limsup \{ \langle \nabla f(x), v \rangle \ | \ x \to \infty \textrm{ and } x\not \in \Omega_f\}.
\end{eqnarray*}\end{corollary}

\begin{example}{\rm 
The function $f \colon \mathbb{R} \to \mathbb{R}, x \mapsto -|x|,$ is 1-Lipschitz and so is Lipschitz at infinity. By Proposition~\ref{Proposition53}, we have
\begin{eqnarray*}
\partial (-f)(\infty)  &=& -\partial f(\infty) \ = \ \mathrm{co}\{-1, 1\}\ = \ [-1, 1]. 
\end{eqnarray*}
}\end{example}

A nonzero vector $v \in \mathbb{R}^n$ is {\em perpendicular} to a set $C \subset \mathbb{R}^n$ at a point $x$ in $\mathrm{cl} C$ (this is denoted $v \perp C$ at $x$) if $v = x' - x,$ where the point $x'$ has unique closest point $x$ in $\mathrm{cl} C.$ Equivalently, $v = x' - x,$ where there is a closed ball centered at $x'$ which meets $\mathrm{cl} C$ only at $x.$ (Note that $x' \not \in \mathrm{cl} C$ because $v \ne 0$.) With this notation, the subgradients at infinity of the distance function $d_C(\cdot)$ can be characterized as follows.

\begin{proposition}
Let $C \subset \mathbb{R}^n$ be a nonempty set. Then
\begin{eqnarray*}
\partial d_{C} (\infty) &=&
\begin{cases}
\mathbb{B} &\textrm{ if } C \textrm{ is bounded}, \\
\mathrm{co}\left( \{0\} \cup A \right) & \textrm{ otherwise,}
\end{cases}
\end{eqnarray*}
where
\begin{eqnarray*}
A &:=& \left\{ v = \lim \dfrac{v_k}{\|v_k\|} \mid v_k = x_k' - x_k \perp C \ {\rm at} \ x_k, \ x_k' \to \infty \right\}.
\end{eqnarray*}
\end{proposition}

\begin{proof}
In view of \cite[Proposition~2.5.4]{Clarke1990}, for each $x \in \mathbb{R}^n,$ if $\nabla d_C (x)$ exists, then it is either zero or the unit vector $\dfrac{x-c}{\|x-c\|}$, where $c$ is the unique closest point in $\mathrm{cl} C$ to $x.$ It follows from Proposition~\ref{Proposition53} that 
\begin{eqnarray*}
\partial d_C(\infty) & \subset & \mathrm{co} \left( \{0\} \cup A \right) \ \subset \ \mathbb{B}.
\end{eqnarray*}
We now prove the desired formula. There are two cases to be considered.

\subsubsection*{Case 1: The set $C$ is bounded.}
Since $\partial d_C(\infty)$ is a convex set and contained in the ball $\mathbb{B},$ it suffices to show that $\mathbb{S}^{n - 1} \subset \partial d_C(\infty).$ To this end, take any $v \in \mathbb{S}^{n - 1}.$ By the classical Weierstrass theorem, there exists a point $x \in \mathrm{cl} C$ such that 
\begin{eqnarray*}
\langle v, {x} \rangle &\ge& \langle v, c \rangle \quad \textrm{ for all } \quad c \in \mathrm{cl} C.
\end{eqnarray*}
It follows that
\begin{eqnarray*}
d_C({x} + \lambda v) &=& \|{x} + \lambda v - {x}\| \ = \ \lambda \quad \textrm{ for all } \quad \lambda \ge 0.
\end{eqnarray*}
Let $0 < \lambda_1 < \lambda_2.$ Applying the mean value theorem (see \cite[Theorem~1.7]{Lebourg1979} or \cite[Theorem~2.3.7]{Clarke1990}) to the function $d_C(\cdot)$, we find $u \in (x + \lambda_1 v, x + \lambda_2 v)$ and $\xi \in \partial d_C(u) $ such that
\begin{eqnarray*}
d_C(x + \lambda_2 v) - d_C(x + \lambda_1 v) &=& \langle \xi, (\lambda_2 - \lambda_1) v \rangle.
\end{eqnarray*}
Therefore,
\begin{eqnarray*}
1 &=& \langle \xi, v \rangle.
\end{eqnarray*}
Since $d_C(\cdot)$ is 1-Lipschitz, $\| \xi\| \le 1.$ Consequently, $v = \xi \in \partial d_C(u).$ Letting $\lambda_1 \to \infty,$ we see that $u \to \infty,$ and so $v \in \partial d_C(\infty)$ due to Corollary~\ref{Corollary43}. Since $v$ was arbitrary in $\mathbb{S}^{n - 1},$ we conclude that $\mathbb{S}^{n - 1} \subset \partial d_{\Omega} (\infty),$ as required.

\subsubsection*{Case 2: The set $C$ is unbounded.} In this case there exists a sequence $c_k \in C$ tending to infinity. Then $d_C(c_k) = 0$ for all $k$. By Proposition~\ref{Proposition44}, $0 \in \partial d_C(\infty).$ Since $\partial d_C(\infty) \subset \mathrm{co} \left( \{0\} \cup A \right),$ as we have seen above, it remains only to verify that $v \in \partial d_C(\infty)$ for all $v \in A.$ To see this, let $v_k := x_k' - x_k \ne 0,$ where $x_k'$ has the closest point $x_k$ in $\mathrm{cl} C$ and $x_k' \to \infty.$ Choose a sequence $t_k \in (0, 1)$ such that $t_k \|x_k' - x_k\| \to 0$ and let $x_k'' := t_k x_k + (1 - t_k) x_k'.$ Then $d_C(x_k'') = \|x_k'' - x_k\|$ and $\|x_k' - x_k''\|  = t_k \|x_k' - x_k\| \to 0.$ Applying the mean value theorem (see \cite[Theorem~1.7]{Lebourg1979} or \cite[Theorem~2.3.7]{Clarke1990}) to the function $d_C(\cdot)$, we find $u_k \in (x_k', x_k'')$ and $\xi_k \in \partial d_C(u_k) $ such that
\begin{eqnarray*}
d_C(x_k') - d_C(x_k'')  &=& \langle \xi_k, x_k' - x_k'' \rangle.
\end{eqnarray*}
As $d_C(x_k') - d_C(x_k'')  = \|x_k' - x_k''\|,$ we get
\begin{eqnarray*}
1 &=& \left \langle \xi_k, \dfrac{x_k' - x_k''}{\|x_k' - x_k''\|}  \right\rangle \ = \ \left \langle  \xi_k, \dfrac{v_k}{\|v_k\|} \right \rangle.
\end{eqnarray*}
Since $d_C(\cdot)$ is 1-Lipschitz, $\| \xi_k\| \le 1.$ Therefore, $\xi_k = \dfrac{v_k}{\|v_k\|}  \in \partial d_C(u_k).$ Noting that $u_k \to \infty$ (since $x_k' \to \infty$ and $\|x_k' - x_k'' \|\to 0$), we deduce from Corollary~\ref{Corollary43} that
\begin{eqnarray*}
v & = & \lim \dfrac{v_k}{\|v_k\|} \ \in \ \partial d_C(\infty),
\end{eqnarray*}
which completes the proof.
\end{proof}

\begin{remark}{\rm
The proof of the above proposition also shows that if $C$ is unbounded, then a vector $v \in \mathbb{R}^n$ belongs to $\partial d_C(\infty)$ if and only if there exist
sequences $x_k \in \mathrm{cl} C$ and $x_k' \not \in \mathrm{cl} C$ with $d_C(x_k') = \|x_k' - x_k \|$ such that $x_k' \to \infty$ and
$\dfrac{x_k' - x_k}{\|x_k' - x_k\|}  \to v.$
}\end{remark}

The third of our main theorems provides conditions guaranteeing that a lower semi-continuous function is Lipschitz at infinity.

\begin{theorem}\label{Theorem57}
Let $f$ be lower semi-continuous. The following statements are equivalent:
\begin{enumerate}[{\rm (i)}]
\item $f$ is Lipschitz at infinity.
\item $\partial f(\infty)$ is nonempty compact.
\item $f^{\uparrow}(\infty; v)$ is finite for all $v \in \mathbb{R}^n.$
\item $f^{\dagger}(\infty; v)$ is finite for all $v \in \mathbb{R}^n.$
\item $f^0(\infty; v)$ is finite for all $v \in \mathbb{R}^n.$
\item $\partial^{\infty} f(\infty)=\{0\}$, where $\partial^{\infty} f(\infty) := \{\xi \in \mathbb{R}^n \ | \ (\xi, 0) \in N_{\mathrm{epi} f}(\infty_I)\}.$
\end{enumerate}
\end{theorem}

\begin{proof} 
We will prove the implications (i) $\Rightarrow$ (v) $\Rightarrow$ (vi) $\Rightarrow$ (ii) $\Rightarrow$ (iii) $\Rightarrow$ (iv) $\Rightarrow$ (i). 

(i) $\Rightarrow$ (v). This follows directly from the definition of $f^0(\infty; \cdot)$. 

(v) $\Rightarrow$ (vi). Take any $v \in \mathbb{R}^n.$ We have $f^{\uparrow}(\infty; v) \leq f^0(\infty; v) < +\infty,$ and so $(v, f^0(\infty; v)) \in T_{\mathrm{epi} f} (\infty_I)$  in view of Proposition~\ref{Proposition45}. Consequently, $\pi(T_{\mathrm{epi} f}(\infty_I)) = \mathbb{R}^n.$ 

We now let $\xi \in \partial^{\infty} f(\infty),$ that is $(\xi, 0) \in N_{\mathrm{epi} f}(\infty_I).$ By definition, then
\begin{eqnarray*}
\langle \xi, v \rangle &\le& 0 \quad \textrm{ for all } \quad v \in \pi(T_{\mathrm{epi} f}(\infty_I)) = \mathbb{R}^n,
\end{eqnarray*}
which yields $\xi = 0.$ Therefore, $\partial^{\infty} f(\infty) = \{0\}.$

(vi) $\Rightarrow$ (ii). We first claim that {\em $N_{{\rm epi}f}(\infty_{I})$ contains a non-zero vector.}

Indeed, if the interior of $T_{{\rm epi}f}(\infty_{I})$ is empty,  then there exists a vector $w \in \mathbb{R}^{n + 1}$ which does not lie in  $T_{{\rm epi}f}(\infty_I).$ Applying the separating hyperplane theorem (see, for example, \cite{Rockafellar1970}) for $w$ and the convex cone $T_{{\rm epi}f}(\infty_I),$ we get a non-zero vector $z \in \mathbb{R}^{n+1}$ such that
\begin{eqnarray*}
\langle z, w' \rangle &\leq& 0 \ \leq \ \langle z, w \rangle \quad \textrm{for all } \quad w' \in T_{\mathrm{epi}f}(\infty_{I}).
\end{eqnarray*}
This implies $z \in N_{{\rm epi}f}(\infty_{I}).$

We now assume that the interior of $T_{{\rm epi}f}(\infty_{I})$ is nonempty. Let $x_k$ be a sequence tending to infinity. Since $(x_k, f(x_k))$ belongs to the boundary of $\mathrm{epi} f,$ it follows from \cite[Corollary~2, p.~67]{Clarke1990} that $N_{{\rm epi}f}((x_k, f(x_k))$ contains a unit vector, say, $z_k.$ Passing to a subsequence if necessary, we may assume that the sequence $z_k$ converges to a unit vector $z.$ Note that $\mathrm{epi} f$ is closed because $f$ is lower semi-continuous. Hence, by Corollary~\ref{Corollary37}, $z \in N_{{\rm epi}f}(\infty_{I}).$

We have proved that $N_{{\rm epi}f}(\infty_{I})$ contains a non-zero vector. Therefore,
$$\partial f (\infty) \cup (\partial^{\infty}f(\infty)\setminus\{0\}) \neq \emptyset.$$
By assumption, hence $\partial f(\infty)$ is non-empty. 

Finally, we show the boundedness of $\partial f(\infty).$ By contradiction, suppose that there is a sequence $\xi_k \in \partial f(\infty)$ tending to infinity. By definition, $(\xi_k, -1)\in N_{{\rm epi}f}(\infty_{I})$, and hence
$$\dfrac{1}{\|\xi_k\|}(\xi_k,-1)\in N_{{\rm epi}f}(\infty_{I}).$$
Passing to a subsequence if necessary, we may assume that $\dfrac{\xi_k}{\|\xi_k\|} \to \xi.$ Then $\xi \ne 0$ and $(\xi, 0) \in N_{{\rm epi}f}(\infty_{I}),$ which contradicts the assumption that $\partial^{\infty} f(\infty)=\{0\}.$

(ii) $\Rightarrow$ (iii). This follows directly from Corollary~\ref{Corollary47}. 

(iii) $\Rightarrow$ (iv). The function $f^{\uparrow}(\infty; \cdot)$ is convex and finite on $\mathbb{R}^n,$ so it must be continuous (see \cite[Corollary~10.1.1]{Rockafellar1970}). Hence, every point $(v, r) \in \mathbb{R}^n \times \mathbb{R}$ with $f^{\uparrow}(\infty; v) < r$ belongs to $\mathrm{epi} f^{\uparrow}(\infty; \cdot)$ and hence belongs also to $T_{\mathrm{epi} f}(\infty_I)$ in view of Proposition~\ref{Proposition45}. This, together with Propositions~\ref{Proposition49} and \ref{Proposition412}, gives $f^{\dagger}(\infty; v) = f^{\uparrow}(\infty; v)$ for all $v \in \mathbb{R}^n.$ 

(iv) $\Rightarrow$ (i). We first show that there are constants $L > 0, $ $R > 0,$ and $\delta >0$ such that for all $x, x' \in \mathbb{R}^n \setminus \B_{R}$ with $\|x - x'\| < \delta,$
\begin{eqnarray*}
|f(x) - f(x')| &\le& L\|x - x'\|.
\end{eqnarray*}
Indeed, if this is not true, then we can find sequences $x_k, x_k'$ tending to infinity such that $t_k := \|x_k - x_k'\|$ tends to zero and that
\begin{eqnarray*}
[f(x_k) - f(x_k')] &=& |f(x_k) - f(x_k')|  \ > \ k \|x_k - x_k'\|.
\end{eqnarray*}
Passing to subsequences if necessary, we may assume that $v_k := \frac{x_k - x_k'}{\|x_k - x_k'\|}$ converges to a vector $v \in \mathbb{R}^n.$ Note that $t_k > 0$ and
\begin{eqnarray*}
\frac{[f(x_k' + t_k v_k) - f(x_k')]}{t_k} &=& \frac{[f(x_k) - f(x_k')]}{t_k}  \ > \ k,
\end{eqnarray*}
which yields $f^{\dagger}(\infty; v) = +\infty,$ a contradiction.

There is a $C^{\infty}$-function $\varphi \colon \mathbb{R}^n \to [0, 1]$ such that 
$$\varphi(x) = \begin{cases}
1 & \textrm{ if } \ \|x\| \ge 3R, \\
0 & \textrm{ if } \ \|x\| \le 2R.
\end{cases}$$
Define the function $\overline{f} \colon \mathbb{R}^n \to \mathbb{R}$ by $\overline{f}(x) := f(x) \varphi(x).$ Then $\overline{f}  \equiv f$ on 
$\mathbb{R}^n\setminus \mathbb{B}_{3R}.$ Furthermore, it is not hard to see that $\overline{f}$ is locally Lipschitz and so it is globally Lipschitz on the compact set $\mathbb{B}_{3R}.$ Increasing $L$ if necessary, we may assume that $\overline{f}$ is locally Lipschitz with constant $L.$
By the property of the Clarke subdifferential, we have $\|\xi\| \leq L$ for all $\xi \in \partial \overline{f}(x)$ and $x \in \mathbb{R}^n.$
We now show that $f$ is Lipschitz with constant $L$ on $\mathbb{R}^n\setminus \mathbb{B}_{3R}.$ 
To see this, take any $x, x' \in \mathbb{R}^n \setminus \B_{3R}.$ By the mean value theorem (see \cite[Theorem~1.7]{Lebourg1979} or \cite[Theorem~2.3.7]{Clarke1990}), we find a point $u \in (x, x')$ and a vector $\xi \in \partial \overline{f}(u)$ such that $\overline{f}(x) - \overline{f}(x') = \langle \xi, x - x'\rangle.$ Hence
\begin{eqnarray*}
|f(x) - f(x')| &=& |\overline{f}(x) - \overline{f}(x')| \ = \ |\langle \xi, x - x' \rangle| \ \le \ L\|x - x'\|,
\end{eqnarray*}
which yields (i).
\end{proof}

The following examples show that smooth functions are not necessarily Lipschitz at infinity and that the set of subgradients of $f$ at infinity may not be a singleton set when $f$ is smooth and Lipschitz at infinity.

\begin{example}\label{Example58} {\rm 
(i) The function $f(x) := e^x$ is of class $C^\infty$ on $\mathbb{R}.$ We have $\partial f(\infty) = [0, +\infty),$ which is unbounded. By Theorem~\ref{Theorem57}, $f$ is not Lipschitz at infinity. (We can also check directly that there is no constant $L > 0$ so that the inequality
$$|e^{k}-e^{k - 1}|\le L|k - (k - 1)|$$
holds for all sufficient large $k.)$

\noindent
(ii) Let $n := 1$ and consider the continuously differentiable function
$$f(x) := 
\begin{cases}
x &\textrm{ if } x \ge 1, \\
\frac{1}{2} x^2+\frac{1}{2} &\textrm{ if } -1 \le x \le 1,\\
-x &\textrm{ if } x \le -1.
\end{cases}$$
Clearly, $f$ is Lipschitz at infinity. In view of Proposition~\ref{Proposition53}, $\partial f(\infty) = [-1, 1],$ which is not a singleton set.
}\end{example}

\subsection*{Acknowledgments}
We would like to thank the anonymous referees for their careful reading with constructive comments and suggestions on the paper. A part of this work was completed during a research stay of the first author at Vietnam Institute for Advanced Study in Mathematics (VIASM); he is warmly grateful to this institute for its hospitality and support.


\end{document}